\documentclass[12pt]{amsart}

\usepackage[colorlinks=true,linkcolor=blue]{hyperref}
\usepackage{amsmath}
\usepackage{amssymb}
\usepackage{amsthm}
\usepackage{bm} % to get boldface Greek letters with \bm{\alpha}, etc.

\usepackage[T1]{fontenc} % to get correct underscores in email address

\usepackage{graphicx}
\usepackage{epstopdf}
\usepackage{epsfig}

\usepackage{tikz}
\usepackage{tikz-cd}

\theoremstyle{plain} 
\newtheorem{thm}{Theorem}[section]

\newtheorem{prop}[thm]{Proposition}
\newtheorem{lemma}[thm]{Lemma}
\newtheorem{cor}[thm]{Corollary}

\theoremstyle{remark}
\newtheorem{remark}[thm]{Remark}

\newtheorem{example}[thm]{Example}

\theoremstyle{definition}
\newtheorem{defin}[thm]{Definition}

\newcommand{\eps}{\varepsilon}

\newcommand{\FF}{\mathbb{F}}

\newcommand{\PP}{\mathbb{P}}

\newcommand{\ZZ}{\mathbb{Z}}

\newcommand{\Kbar}{\overline{K}}

\newcommand{\PK}{\PP^1(K)}
\newcommand{\PKbar}{\PP^1(\Kbar)}

\DeclareMathOperator{\CR}{CR}
\DeclareMathOperator{\sgn}{sgn}
\DeclareMathOperator{\PGL}{PGL}

\DeclareMathOperator{\Orb}{Orb}

\DeclareMathOperator{\charact}{char}
\DeclareMathOperator{\Res}{Res}

\DeclareMathOperator{\Gal}{Gal}

\DeclareMathOperator{\Aut}{Aut}

\newcommand{\dsps}{\displaystyle}

\begin{document}

\title[Arboreal Galois groups and colliding critical points]
{Arboreal Galois groups for quadratic rational functions
with colliding critical points}
%\date{July 30, 2023} %SUBMITTED VERSION
%\date{April 5, 2025} %MINOR CORRECTION (BIBLIOGRAPHY)
%\date{May 17, 2024} %ANNA'S FIRST-PASS POST-REFEREE REVISIONS
\date{July 30, 2023; revised June 1, 2024}
\subjclass[2020]{37P05, 11R32, 14G25}
\author[Benedetto]{Robert L. Benedetto}
\address{Amherst College \\ Amherst, MA 01002 \\ USA}
\email{rlbenedetto@amherst.edu}
\author[Dietrich]{Anna Dietrich}
\address{Brown University \\ Providence, RI 02912 \\ USA}
\email{anna\_dietrich@brown.edu}
%\email{adietrich22@amherst.edu}

\begin{abstract}
Let $K$ be a field, and let $f\in K(z)$ be rational function.
The preimages of a point $x_0\in\PK$ under iterates of $f$ have a
natural tree structure. As a result, the Galois group of the resulting
field extension of $K$ naturally embeds into the automorphism group
of this tree. In unpublished work from 2013, Pink described a certain
proper subgroup $M_{\ell}$ that
this so-called arboreal Galois group $G_{\infty}$ must lie in
if $f$ is quadratic and its two critical points collide at the $\ell$-th iteration.
After presenting a new description of $M_{\ell}$ and a new proof of Pink's theorem,
we state and prove necessary and sufficient conditions for $G_{\infty}$ to be
the full group $M_{\ell}$.
\end{abstract}

\maketitle

\section{Introduction}

Let $K$ be a field, and let $f\in K(z)$ be a rational function.
Writing $f=g/h$ with $g,h\in K[z]$ relatively prime,
we define the degree of $f$ to be $\deg f = \max\{\deg g, \deg h\}$.
Then $f$ induces an endomorphism $f:\PKbar\to\PKbar$, where
$\Kbar$ is an algebraic closure of $K$. If $d\geq 1$,
then every point of $\PKbar$
has $d=\deg f$ preimages in $\PKbar$, counted with multiplicity.

For any integer $n\geq 0$, we write $f^n=f\circ\cdots\circ f$
for the $n$-th iterate of $f$ under composition,
with $f^0(z)=z$ and $f^1(z)=f(z)$.
For any point $x_0\in\PKbar$,
the \emph{forward orbit} and
\emph{backward orbit} of $x_0$ under $f$ are
\[ \Orb_f^{+}(x_0) := \{f^n(x_0) \, | \, n\geq 0\}
\quad\text{and}\quad
\Orb_f^{-}(x_0) := \coprod_{n\geq 0} f^{-n}(x_0) \subseteq \PKbar, \]
respectively,
where $f^{-n}(y)$ denotes the set $(f^n)^{-1}(y)$ of solutions
of the equation $f^n(z)=y$ in $\PKbar$.
Unless otherwise stated,
we will assume that there are no critical points (i.e., ramification points)
in the backward orbit of $x_0$,
%so that $f^{-n}(y)$ has $d^n$ elements, counted with multiplicity.
so that $f^{-n}(x_0)$ has $\deg(f^n)=d^n$ elements.

For each $n\geq 0$, define
\[ K_n:=K(f^{-n}(x_0))\subseteq\Kbar .\]
If $f$ is a separable mapping, then $K_n/K$ is a separable
and hence Galois extension, and we define
\[ G_n:=\Gal(K_n/K) \]
to be the associated Galois group. We also define
\[ G_{\infty}:=\Gal(K_{\infty}/K)\cong\varprojlim G_n, \quad\text{where}\quad  K_{\infty}:=\bigcup_{n\geq 0} K_n .\]

Assuming there are no critical points
of $f$ in the backward orbit of $x_0$,
we may consider $\Orb_f^-(x_0)$ as forming
an infinite $d$-ary rooted tree $T_{d,\infty}$.
The root node of the tree is $x_0$,
the elements of $f^{-n}(x_0)$ are the $d^n$ nodes on the $n$-th level
of the tree, and each $y\in f^{-(n+1)}(x_0)$
is connected to $f(y)\in f^{-n}(x_0)$ by an edge.
After making this identification, then, $G_{\infty}$
is isomorphic to a subgroup of the automorphism group $\Aut(T_{d,\infty})$
of the tree. (Here and throughout this paper, when we say that two groups
acting on a tree are isomorphic, we mean that the isomorphism is
equivariant with respect to the action on the tree.)
Similarly, $G_n$ is isomorphic
to a subgroup of $\Aut(T_{d,n})$ for each $n\geq 0$,
where $T_{d,n}$ is rooted $d$-ary tree up to the $n$-th level.
In light of this action on trees, the groups $G_n$ and $G_{\infty}$
have come to be known as \emph{arboreal} Galois groups.
See
\cite{BosJon,BEK21,BHL,FPC,Hindes,JonMan,Swam}
%\cite{AHM,FM,FPC,Hindes2,JKMT,Stoll}
for a limited selection of results on this topic,
and \cite{Jones} for a survey of the field.

When $K$ is a number field or function field,
it has been shown that $G_{\infty}$
can be the full group $\Aut(T_{d,\infty})$
for some choices of $K$, $f$, $x_0$;
see, for example,
\cite{BenJuu,Juul,Kadets,Looper,Odoni,Specter,Stoll}.
In analogy with Serre's Open Index Theorem for Galois
representations arising from elliptic curves \cite{Serre},
a folklore conjecture states that
when $K$ is a number field or function field,
$G_{\infty}$ should usually have finite index
in $\Aut(T_{d,\infty})$.
Indeed, Jones formulated this statement
as a precise conjecture for $d=2$ in \cite[Conjecture~3.11]{Jones},
and there are some conditional results for $d=2,3$
in \cite{BriTuc,JKetal}.

Just as Serre's Theorem has an exception for CM curves,
these conjectures and results have exceptions for
certain situations in which the index
$[\Aut(T_{d, \infty}):G_{\infty}]$ is known to be infinite.
For example, if $f(z)= z^d+c$ with $d\geq 3$, then the functional
equation $f(\zeta_d z)=f(z)$, where $\zeta_d$ is a $d$-th root of
unity, yields extra symmetries in $\Orb_f^-(x_0)$ and hence
infinite index, as described in \cite{BHL}.
Another case is that $f$ is
\emph{postcritically finite}, or PCF,
meaning that every critical point $c$ is preperiodic,
i.e., there exist integers $n>m\geq 0$ such that $f^n(c)=f^m(c)$.
(See \cite[Theorem~3.1]{Jones} for a proof that
$[\Aut(T_{d, \infty}):G_{\infty}]=\infty$ for PCF maps,
and \cite{ABCCF,BFHJY,Pink} for the arboreal Galois groups
of certain PCF examples.)
In this paper, we consider another condition
which forces $[\Aut(T_{d, \infty}):G_{\infty}]=\infty$, as follows.

\begin{defin}
\label{def:collide}
Let $f\in K(z)$ be a rational function,
let $\xi_1,\xi_2\in\PKbar$ be two critical points of $f$,
and let $\ell\geq 1$ be a positive integer.
We say that $\xi_1$ and $\xi_2$ \emph{collide}
at the $\ell$-th iterate if
\begin{equation}
\label{eq:PinkCond}
f^{\ell}(\xi_1)=f^{\ell}(\xi_2)
\quad\text{but } f^{\ell-1}(\xi_1)\neq f^{\ell-1}(\xi_2).
\end{equation}
\end{defin}

If $\charact K\neq 2$, and if $f$ is quadratic with colliding critical points,
then Pink observed in \cite[Theorem~4.8.1]{Pink2} that
$[\Aut(T_{2,\infty}):G_{\infty}]=\infty$.
(In this case, we must have $\ell\geq 2$,
because the two critical values must be distinct,
as each has only $d=2$ preimages, counting multiplicity.)
More precisely, if $K=k(t)$ where $k$ is a field with $\charact k\neq 2$,
and if $f\in k(z)$ and $x_0=t$,
Pink described the resulting Galois group $G_{\infty}$
in terms of a countable set of topological generators in $\Aut(T_{2,\infty})$.
If a certain discriminant is not a square in $k$, 
then $G_{\infty}$ is isomorphic to a certain subgroup
$\widetilde{M}_{\ell}$ of $\Aut(T_{2,\infty})$,
and otherwise $G_{\infty}$ is isomorphic to an index~2 subgroup
$M_{\ell}$ of $\widetilde{M}_{\ell}$.
(We define $\widetilde{M}_{\ell}$ and $M_{\ell}$ in Definition \ref{def:PinkGroup}.
In \cite{Pink2}, Pink denotes these groups $\widetilde{G}(r)$ and $G(r)$, respectively,
where $r=\ell-1$.)
Specializing from $K=k(t)$ to the constant field $k$,
the arboreal Galois group for such $f$ with root point $x_0\in k$
is therefore a subgroup of $\widetilde{M}_{\ell}$.

Our goal in this paper is twofold.
First, we describe the groups $\widetilde{M}_{\ell}$ and $M_{\ell}$ as sets of
automorphisms $\sigma\in\Aut(T_{2,\infty})$ satisfying a certain condition,
rather than by giving generators, as Pink did.
Second, we present necessary and sufficient conditions for the
arboreal Galois group $G_\infty$ associated with
a quadratic rational map $f\in K(z)$ satisfying hypothesis~\eqref{eq:PinkCond}
to be isomorphic to the full group $\widetilde{M}_{\ell}$ or $M_{\ell}$.
These goals are summarized in the following two theorems.

%\cite[Conjecture~2]{ABCCF} says this:
%
%\begin{conj}
%\label{conj:general}
%Let $k$ be a number field, and
%let $\phi(z)\in k(z)$ be a rational function of degree $d\geq 2$,
%defined over $k$.
%Let $L$ be the function field $L:=k(t)$,
%and let $G(\phi,k)$ be the arboreal Galois group of $\phi$ over $L$
%with root point $t$; that is,
%%%%%%viewed as a subgroup of $\Aut(T_{d,\infty})$. That is,
%\[ G(\phi,k) :=\Gal(L_{\infty}/L), \quad \text{where} \quad
%L_{\infty} :=\bigcup_{n\geq 0} L\big(\phi^{-n}(t)\big) .\]
%Then for any finite extension $K/k$ and any $x_0\in\PP^1(K)$,
%the associated arboreal Galois group $G_{\infty}:=\Gal(K_{\infty}/K)$
%is isomorphic to a subgroup of $G(\phi,k)$.
%Moreover, with $K=k$, it is possible to choose $x_0$
%so that $G_{\infty}$ is the full group $G(\phi,k)$.
%\end{conj}

\begin{thm}
\label{thm:main1}
Let $K$ be a field of characteristic different from $2$, and
let $f\in K(z)$ be a rational function of degree $2$ with critical points
$\xi_1,\xi_2\in\PKbar$.
Let $\delta\in K^{\times}$ be the discriminant of the minimal polynomial
of $\xi_1$ over $K$,
which we understand to be $\delta=1$ if $\xi_1\in\PK$.
Fix $x_0\in\PK$, and let $G_{\infty}$ be the arboreal Galois group for
$f$ over $K$, rooted at $x_0$.
%Suppose that $f^{\ell}(\xi_1)=f^{\ell}(\xi_2)$ and $f^{\ell-1}(\xi_1)\neq f^{\ell-1}(\xi_2)$
%for some integer $\ell\geq 2$.
Suppose that $\xi_1$ and $\xi_2$ collide at the $\ell$-th iterate under $f$,
for some integer $\ell\geq 2$. Then:
\begin{enumerate}
\item $G_{\infty}$ is isomorphic to a subgroup of $\widetilde{M}_{\ell}$,
via an appropriate labeling of the tree.
\item $G_{\infty}$ is isomorphic to a subgroup of $M_{\ell}$
if and only if $\delta$ is a square in $K$.
\end{enumerate}
\end{thm}

\begin{thm}
\label{thm:main2}
With notation and hypothesis as in Theorem~\ref{thm:main1},
%suppose that the critical points $\xi_1,\xi_2$ are not preperiodic. Then
there is a countable sequence
%of quantities $\kappa_1,\kappa_2,\ldots\in K$
of quantities
\[ \kappa_1,\kappa_2,\ldots\in L:=K(\sqrt{\delta}) ,\]
given by explicit expressions involving $f$ and $x_0$,
with $\kappa_1,\ldots,\kappa_{\ell-1}\in K$,
so that the following hold.
\begin{enumerate}
\item If $\delta$ is a square in $K$, then the following are equivalent:
\begin{enumerate}
\item No finite product $\dsps \kappa_{i_1} \cdots \kappa_{i_m}$
(for $1\leq i_1 < \cdots < i_m$ and $m\geq 1$) is a square in $K$.
\item $G_\infty\cong M_{\ell}$.
\end{enumerate}
\item If $\delta$ is not a square in $K$, then $\kappa_\ell \delta$ is
a square in $K$, and the following are equivalent:
\begin{enumerate}
\item The only finite product $\dsps \kappa_{i_1} \cdots \kappa_{i_m}$
(for $1\leq i_1 < \cdots < i_m$ and $m\geq 1$) that is a square in $L$
%$L:=K(\sqrt{\delta})$,
is the single element $\kappa_{\ell}$.
\item $G_\infty\cong \widetilde{M}_{\ell}$.
\end{enumerate}
\end{enumerate}
\end{thm}
We define the quantities $\kappa_n\in K$ of Theorem~\ref{thm:main2} in Definition~\ref{def:kappa}.

As noted earlier, the finite Galois group $G_n=\Gal(K_n/K)$ acts on the finite tree $T_{2,n}$,
for any integer $n\geq 0$.
It is therefore convenient to define
subgroups $M_{\ell,n}$ and $\widetilde{M}_{\ell,n}$ of the finite group $\Aut(T_{2,n})$,
restricting elements of $M_{\ell}$ and $\widetilde{M}_{\ell}$ to $T_{2,n}$.
(See Definition~\ref{def:Mfinite}.)
We present analogs of Theorems~\ref{thm:main1} and~\ref{thm:main2}
for this finite setting in Corollary~\ref{cor:main1} and Theorem~\ref{thm:main2big},
respectively.

%For ease of navigation, the proof of Theorem \ref{thm:main1} is on page \pageref{prf: main1},
%and the proof of Theorem \ref{thm:main2} is on page \pageref{prf: main2}.

The outline of the paper is as follows.
In Section~\ref{sec:prelim}, we set notation, describe tree labelings,
and define the groups $M_{\ell}$ and $\widetilde{M}_{\ell}$ in terms of
the parities of $\sigma\in\Aut(T_{2,\infty})$ acting on various portions of the tree.
In Section~\ref{sec:discrim}, we present a number of elementary results
involving discriminants, most notably a formula (Corollary~\ref{cor:iterdisc})
for discriminants of iterated maps in homogeneous coordinates.
We then apply these formulas to prove Theorem~\ref{thm:main1}
in Section~\ref{sec:collideM}, using the discriminants of Section~\ref{sec:discrim}
to detect the parity conditions of Section~\ref{sec:prelim}.
In Section~\ref{sec:cousins}, we study certain generators for
the groups $M_{\ell}$ and $\widetilde{M}_{\ell}$, in particular tree automorphisms
that we call \emph{odd cousins maps}.
Finally, in Section~\ref{sec:surj}, we define the quantities $\kappa_n$
of Theorem~\ref{thm:main2} in terms of cross ratios and discriminants
involving iterates of $f$. In Lemmas~\ref{lem:Qn} and~\ref{lem:Rn},
we present explicit algebraic expressions involving iterated preimages of $x_0$,
illustrated in Figures~\ref{fig:QnLemma} and~\ref{fig:RnLemma}, that can detect
whether or not a given Galois automorphism is an odd cousins map.
We then apply Lemmas~\ref{lem:Qn} and~\ref{lem:Rn} to prove Theorem~\ref{thm:main2}.

\section{Preliminaries}
%\section{Notation, tree labelings, and Pink's groups}
\label{sec:prelim}

\subsection{Notation and tree labelings}
We set the following notation throughout this paper.
\begin{tabbing}
\hspace{8mm} \= \hspace{15mm} \=  \kill
\> $K$: \> a field of characteristic different from $2$,
with algebraic closure $\Kbar$ \\
\> $\ell$: \> an integer $\ell\geq 2$ \\
\> $f$: \> a rational function $f(z) \in K(z)$, usually of degree~$2$ \\
%\> $\xi_1,\xi_2$: \> the two critical points of $f$ in $\PKbar$ \\
\> $x_0$: \> an element of $\PK$, to serve as the root of our preimage tree \\
\> $T_n$: \> a binary rooted tree, extending $n$ levels above its root node \\
\> $T_\infty$: \> a binary rooted tree, extending infinitely above its root node \\
\> $K_n$: \> for each $n\geq 0$, the extension field $K_n:=K(f^{-n}(x_0))$ \\
\> $K_\infty$: \> the union $K_{\infty} = \bigcup_{n\geq 1} K_n$ in $\Kbar$ \\
\> $G_n$: \> the Galois group $\Gal(K_n/K_0)$ \\
\> $G_{\infty}$: \> the Galois group $\Gal(K_\infty/K_0)$
\end{tabbing}
Following \cite[Definition~1.3]{ABCCF}, we assign labels
to each of the nodes of the abstract trees $T_n$ and $T_{\infty}$,
as follows.

\begin{defin}
\label{def:labeling}
A \emph{labeling} of $T_{\infty}$ is a choice of
two tree morphisms $0,1:T_{\infty}\to T_{\infty}$ such that
$0$ maps $T_{\infty}$ bijectively onto the subtree rooted at one
of the two nodes connected to the root node $x_0$,
and $1$ maps $T_{\infty}$ bijectively onto the subtree rooted at the other.

For any integer $n\geq 1$, a labeling of $T_n$ is a choice of
two injective tree morphisms $0,1:T_{n-1}\to T_n$
with the same property.
\end{defin}

Given the two tree morphisms 0,1 of Definition~\ref{def:labeling},
we can assign a label to each node of the tree, as follows.
First, label the root node with the empty word $()$.
Then, for each level $m\geq 1$ of the tree and each node $w$ at level $1$,
label $w$ with the unique ordered $m$-tuple
$s_1s_2\ldots s_m\in\{0,1\}^m$, such that
$w=s_1\circ s_2\circ \cdots\circ s_m()$.
See Figure~\ref{fig:labeltree}.

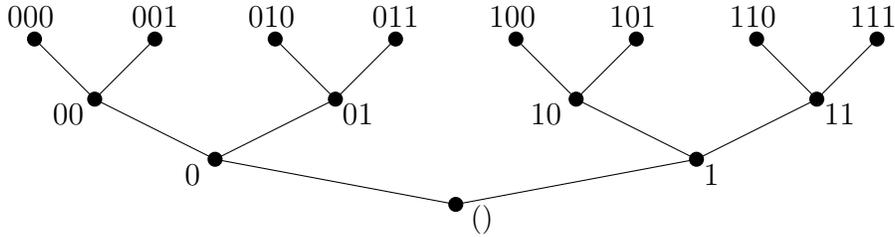
\begin{figure}
\begin{tikzpicture}
\path[draw] (0.8,2.4) -- (1.6,1.6) -- (2.4,2.4);
\path[draw] (4,2.4) -- (4.8,1.6) -- (5.6,2.4);
\path[draw] (7.2,2.4) -- (8,1.6) -- (8.8,2.4);
\path[draw] (10.4,2.4) -- (11.2,1.6) -- (12,2.4);
\path[fill] (0.8,2.4) circle (0.1);
\path[fill] (2.4,2.4) circle (0.1);
\path[fill] (4,2.4) circle (0.1);
\path[fill] (5.6,2.4) circle (0.1);
\path[fill] (7.2,2.4) circle (0.1);
\path[fill] (8.8,2.4) circle (0.1);
\path[fill] (10.4,2.4) circle (0.1);
\path[fill] (12,2.4) circle (0.1);
\path[draw] (1.6,1.6) -- (3.2,0.8) -- (4.8,1.6);
\path[draw] (8,1.6) -- (9.6,0.8) -- (11.2,1.6);
\path[fill] (1.6,1.6) circle (0.1);
\path[fill] (4.8,1.6) circle (0.1);
\path[fill] (8,1.6) circle (0.1);
\path[fill] (11.2,1.6) circle (0.1);
\path[draw] (3.2,0.8) -- (6.4,0.2) -- (9.6,0.8);
\path[fill] (3.2,0.8) circle (0.1);
\path[fill] (9.6,0.8) circle (0.1);
\path[fill] (6.4,0.2) circle (0.1);
\node (x0) at (6.75,0) {\small $()$};
\node (a0) at (2.9,0.6) {$0$};
\node (b0) at (9.8,0.6) {$1$};
\node (a00) at (1.25,1.4) {$00$};
\node (b01) at (5.1,1.4) {$01$};
\node (a10) at (7.6,1.4) {$10$};
\node (b11) at (11.5,1.4) {$11$};
\node (a000) at (0.75,2.7) {$000$};
\node (a001) at (2.4,2.7) {$001$};
\node (a010) at (3.95,2.7) {$010$};
\node (a011) at (5.6,2.7) {$011$};
\node (a100) at (7.15,2.7) {$100$};
\node (a101) at (8.75,2.7) {$101$};
\node (a110) at (10.35,2.7) {$110$};
\node (a111) at (11.95,2.7) {$111$};
\end{tikzpicture}
\caption{Labeling the tree $T_3$}
\label{fig:labeltree}
\end{figure}

Having fixed a labeling,
we will often abuse terminology and conflate a node $x$ of the tree with its label.
In addition, for each point $y\in\Orb^-_f(x_0)$,
we will often label the corresponding node
of the abstract tree $T_n$ or $T_{\infty}$ as $y$ as well,
with $x_0$ as the root node, and the $2^m$ nodes
at the $m$-th level of the tree as the points of $f^{-m}(x_0)$.

%(If there is a critical point in the backward orbit of $x_0$, we will repeat it
%and all of its preimages with appropriate multiplicity in this labeling.)

\subsection{Higher-level signs and Pink's groups}
\label{ssec:sgndef}
Fix a labeling of the tree $T_{\infty}$.
Let $y$ be a node of the tree, and
let $m\geq 1$ be a positive integer.
The $2^m$ nodes that are $m$ levels above $y$ have labels
$y s_1 s_2 \ldots s_m$, with each $s_i\in\{0,1\}$.
For any automorphism $\sigma\in\Aut(T_{\infty})$ of the (rooted) tree,
we have
\[ \sigma(y s_1 s_2 \ldots s_m) = \sigma(y) t_1 t_2 \ldots t_m,
\quad\text{for some } t_1,\ldots,t_m\in\{0,1\}. \]
Thus, $\sigma$ and $y$ together induce a bijective function from $\{0,1\}^m$ to itself,
sending $(s_1, \ldots, s_m)$ to $(t_1, \ldots, t_m)$.
Following Pink,
we define the \emph{$m$-th sign} of $\sigma$ above $y$,
denoted $\sgn_m(\sigma,y)$, to be the sign
of this permutation of  $\{0,1\}^m$
--- that is, $+1$ if the permutation is of even parity, or $-1$ if it is odd.

\begin{defin}
\label{def:PinkGroup}
Fix a labeling of the tree $T_{\infty}$. Let $\ell\geq 2$ be an integer.
We define $\widetilde{M}_{\ell}$ to be the set of all $\sigma\in\Aut(T_{\infty})$ for which
\[ \sgn_{\ell}(\sigma,y) = \sgn_{\ell}(\sigma,x_0)
\quad \text{for every node } y \text{ of } T_{\infty} .\]
We also define $M_{\ell}$ to be the set of all $\sigma\in \widetilde{M}_{\ell}$
for which this common sign is $+1$.
\end{defin}

Note that for any $\sigma\in\Aut(T_{\infty})$ and any node $y$
for which $\sigma(y)\neq y$, the sign $\sgn_\ell(\sigma,y)$ depends on
the choice of labeling of the tree. Thus, the subgroups $\widetilde{M}_{\ell}$ and $M_{\ell}$
also depend on the labeling. However, any two labelings are conjugate by
an automorphism of the tree, so a change of labeling has the effect of
replacing the subgroups $\widetilde{M}_{\ell}$ and $M_{\ell}$ by appropriate conjugates.

\begin{thm}
Fix a labeling of the tree $T_{\infty}$ and an integer $\ell\geq 2$.
Then $\widetilde{M}_{\ell}$ and $M_{\ell}$ are subgroups of $\Aut(T_{\infty})$.
\end{thm}

\begin{proof}
Clearly the identity automorphism $e$ belongs to $M_{\ell}\subseteq\widetilde{M}_{\ell}$.
For any $\sigma,\tau\in\Aut(T_{\infty})$, and for any node $y$ of $T_{\infty}$,
we have
\begin{equation}
\label{eq:Parident}
\sgn_{\ell}(\sigma\tau,y)
= \sgn_{\ell}(\sigma,\tau(y)) \cdot \sgn_{\ell}(\tau,y).
\end{equation}
Indeed, $\tau$ maps the nodes above $y$ to the nodes above $\tau(y)$,
in particular permuting labels of the $2^\ell$ nodes that are $\ell$ levels above $y$
with sign $\sgn_{\ell}(\tau,y)$.
Then $\sigma$ permutes the labels of those same nodes with sign
$\sgn_{\ell}(\sigma,\tau(y))$ while moving them to the nodes above $\sigma(\tau(y))$.

Thus, if $\sigma,\tau \in \widetilde{M}_{\ell}$, then for any node $y$ of the tree, we have
\[ \sgn_{\ell}(\sigma\tau,y) = 
\sgn_{\ell}(\sigma,\tau(y)) \cdot \sgn_{\ell}(\tau,y) =
\sgn_{\ell}(\sigma,\tau(x_0)) \cdot \sgn_{\ell}(\tau,x_0) =
\sgn_{\ell}(\sigma\tau,x_0),\]
and hence $\sigma\tau\in\widetilde{M}_{\ell}$.
Similarly, if $\sigma,\tau\in M_{\ell}$, then all the signs above are $+1$,
so we have $\sigma\tau\in M_{\ell}$.

In addition, for any $\sigma\in \widetilde{M}_{\ell}$, then choosing
$\tau=\sigma^{-1}\in\Aut(T_{\infty})$, equation~\eqref{eq:Parident} yields
\[ \sgn_{\ell}(\sigma^{-1},y) = \sgn_{\ell}(\sigma, \sigma^{-1}(y))
= \sgn_{\ell}(\sigma, \sigma^{-1}(x_0))
= \sgn_{\ell}(\sigma^{-1},x_0), \]
and hence $\sigma^{-1}\in\widetilde{M}_{\ell}$.
Similarly, $M_{\ell}$ is also closed under inverses.
\end{proof}

\subsection{Pink's description of $M_{\ell}$ and $\widetilde{M}_{\ell}$}
\label{ssec:PinkDes}
In Section~4.2 of \cite{Pink2}, Pink defines $G(\ell-1)$
to be the closure of the subgroup of $\Aut(T_{\infty})$
generated by a certain countable set of automorphisms
\[ \{a_i \, | \, i\geq 1 \}\cup\{b_j \, | \, 1\leq j\leq \ell-1\}\subseteq\Aut(T_{\infty}) .\]
Here, we mean closure with respect to the natural topology on $\Aut(T_{\infty})$,
given by the basis $\{ E_n \, | \, n\geq 1\}$ of open neighborhoods of $e$,
where for each $n\geq 1$, the normal subgroup $E_n$ consists of those
$\sigma\in\Aut(T_{\infty})$ that are trivial on the finite subtree $T_{n-1}$.

%Fix a labeling of $T_{\infty}$.
%If $\sigma\in\Aut(T_{\infty})$ fixes both nodes $0$ and $1$ at level~1 of the tree,
%then it is common to write $\sigma=(\sigma_0,\sigma_1)$,
%where $\sigma_0,\sigma_1\in\Aut(T_{\infty})$ are the actions of $\sigma$ on the
%subtrees rooted at the nodes $0$ and $1$, respectively.
%Let $\mapa_1$ denote the element of $\Aut(T_{\infty})$
%that swaps nodes $0$ and $1$ but leaves the rest of any node's label
%unchanged; that is,
%$\mapa(0s_2 s_3 \ldots) = 1 s_2 s_3 \ldots$,
%and $\mapa(1s_2 s_3 \ldots) = 0 s_2 s_3 \ldots$.
%In equation~(4.2.1) of \cite{Pink2}, Pink defines his group $G(\ell+1)$
%to be the subgroup of $\Aut(T_{\infty})$ generated by $\mapa_1$
%and the countably many recursively defined elements
%\[ \mapa_i=\begin{cases}
%(\mapa_{i-1},e) & \text{ if } i\geq 2 \text{ and } i\neq\ell,
%\\
%(\mapa_{\ell-1},\mapb_{\ell-1}) & \text{ if } i=\ell,
%\end{cases}\]
%together with the finitely many further recursively defined elements
%\[ \mapb_j=\begin{cases}
%(\mapb_{\ell-1}, \mapb_{\ell-1}^{-1})\mapa_1 & \text{ if } j=1,
%\\
%(\mapb_{\ell-1}\mapb_{j-1}\mapb_{\ell-1}^{-1}, e) & \text{ if } 2\leq j \leq \ell-1.
%\end{cases}  \]

In Proposition~4.2.5 of \cite{Pink2}, Pink also works out the value of
%$\sgn_m(\mapa_i)$ and $\sgn_m(\mapb_j)$ for all $m\geq 1$,
the signs $\sgn_m(a_i,x_0)$ and $\sgn_m(b_j,x_0)$ for all $m\geq 1$,
all $i\geq 1$, and all $1\leq j\leq\ell-1$.
In particular, combining his sign formulas with his recursive definitions
of the elements $a_i, b_j$, it is immediate
that the signs $\sgn_{\ell}(a_i,y)$ and $\sgn_{\ell}(b_j,y)$ are $+1$ for all
such $i,j$ and all nodes $y$ of the tree.
Thus, all of Pink's generators belong to our group $M_{\ell}$.
Moreover, $M_{\ell}$ is clearly closed with respect to the natural topology
on $\Aut(T_{\infty})$, so it follows that Pink's group
$G(\ell-1)$ is contained in our group $M_{\ell}$.
We claim that the two groups coincide.
To do so, we define the following two subgroups of $\Aut(T_n)$.

\begin{defin}
\label{def:Mfinite}
Let $\ell\geq 2$ and $n\geq 0$ be integers.
Define $M_{\ell,n}$ to be the quotient of $M_{\ell}$
formed by restricting each $\sigma\in M_{\ell}$ to the subtree $T_n$.
Similarly define $\widetilde{M}_{\ell,n}$ to be the quotient of $\widetilde{M}_{\ell}$
formed by restricting to $T_n$.
\end{defin}

It is well known that $\log_2 |\Aut(T_n)|=2^n-1$, since any
$\sigma\in \Aut(T_n)$ can be described by specifying,
for each of the $2^n - 1$ nodes $y$ at levels $0$ to $n-1$,
whether $\sigma$ switches or fixes
the labels of the two nodes immediately above $y$.
For each node $y$ of $T_n$ at level $n-\ell$ or lower, the condition
that $\sgn_{\ell}(\sigma,y)=+1$ introduces an index~2 restriction
on $M_{\ell,n}$, and the restrictions for these various nodes $y$
are independent of one another.
Since there are no such nodes for $n<\ell$, and $2^{n-\ell+1}-1$ such nodes
for $n\geq \ell$, it follows that
\[ \log_2 [\Aut(T_n) : M_{\ell,n} ] =\begin{cases}
0 & \text{ if } 0\leq n\leq \ell-1, \\
2^{n-\ell+1}-1 & \text{ if } n\geq \ell.
\end{cases} \]
(Incidentally, it follows that $M_{\ell}$ is of infinite index in $\Aut(T_{\infty})$.)
Hence,
\begin{equation}
\label{eq:Msize}
\log_2 | M_{\ell,n} | =\begin{cases}
2^n-1 & \text{ if } 0\leq n\leq \ell-1, \\
2^n - 2^{n-\ell+1} & \text{ if } n\geq \ell.
\end{cases}
\end{equation}
This formula exactly coincides with Pink's computation of $|G(\ell-1)_n|$
in Proposition~4.4.1 of \cite{Pink2}, thus proving our claim
that Pink's group $G(\ell-1)$ is the same as our group $M_{\ell}$.

Pink goes on to define a larger group $\widetilde{G}(\ell-1)$ by adding
one more generator, which he denotes $\tilde{w}$.
In equation~(4.8.8) of \cite{Pink2}, he defines
$\tilde{w}$ by a recursive relation involving his earlier generator $b_{\ell-1}$.
It is immediate from this definition and his aforementioned sign computations
that $\sgn_{\ell}(\tilde{w},y)=-1$ for every node $y$ of $T_{\infty}$,
so that $\tilde{w}\in\widetilde{M}_{\ell}$,
and hence $\widetilde{G}(\ell-1)\subseteq \widetilde{M}_{\ell}$.
On the other hand, just after defining $\tilde{w}$, Pink notes that
$[\widetilde{G}(\ell-1):G(\ell-1)]=2$, and of course we clearly have
$[\widetilde{M}_{\ell}:M_{\ell}]=2$ as well.
Therefore, it follows that Pink's group $\widetilde{G}(\ell-1)$
also coincides with our group $\widetilde{M}_{\ell}$.

\begin{remark}
We limit the use of Pink's notation $a_i,b_j,\tilde{w}$ to this section.
Having proven that Pink's groups are the same as our groups
$M_{\ell}$ and $\widetilde{M}_{\ell}$, we will no longer need to refer
to the specific generators Pink defined in \cite{Pink2}.
\end{remark}

\section{Homogeneous discriminants}
\label{sec:discrim}
The discriminants of the equations $f^n(z)=x_0$ are especially
important in the study of arboreal Galois groups.
Iterative formulas for these discriminants may be found,
for example, in \cite[Proposition~3.2]{AHM} for the polynomial case,
and in \cite[Theorem~3.2]{JonMan} for the rational function case.
In our setting, it will be convenient to derive analogous formulas
working in homogeneous coordinates on $\PP^1$, as follows.

\begin{defin}
\label{def:hdisres}
Let $P,Q\in K[X,Y]$ be homogeneous polynomials of degrees $m,n\geq 1$,
respectively. Write
\[ P(X,Y) = \prod_{i=1}^m (b_i X - a_i Y)
\quad\text{and}\quad
Q(X,Y) = \prod_{i=1}^n (d_i X - c_i Y) , \]
with $a_i,b_i,c_i,d_i\in\Kbar$,
such that for each $i$, at least one of $a_i,b_i$ is nonzero,
and at least one of $c_i,d_i$ is nonzero.
The (\emph{homogeneous}) \emph{resultant} of $P$ and $Q$ is
\[ \Res(P,Q) := \prod_{i=1}^m \prod_{j=1}^n (a_i d_j - b_i c_j) \in K, \]
and the (\emph{homogeneous}) \emph{discriminant} of $P$ is
\[ \Delta(P) := \prod_{1\leq i < j \leq m} (a_i b_j - b_i a_j)^2 \in K .\]
\end{defin}

If $P=a$ is constant, we understand $\Res(a,Q)$ to be $a^{\deg Q}$,
and similarly if $Q=c$ is constant, then $\Res(P,c)=c^{\deg P}$.
If $\deg P=1$, we understand $\Delta(P)$ to be $1$.

In the notation of Definition~\ref{def:hdisres},
the points $[a_i:b_i]\in\PKbar$ are the zeros of $P$.
Both the resultant and discriminant are invariant under
all (Galois) permutations of these zeros, and hence they do indeed lie in $K$.
They are also invariant under replacing
$(a_i,b_i)$ by $(\lambda_i a_i,\lambda_i b_i)$
for any $\lambda_1,\ldots,\lambda_m\in \Kbar^{\times}$
satisfying $\lambda_1\cdots\lambda_m=1$.
Clearly, we have
\begin{equation}
\label{eq:ResProd}
\Res(P,Q) = \prod_{i=1}^m Q(a_i,b_i)
= (-1)^{mn} \prod_{i=1}^n P(c_i,d_i).
\end{equation}
If $b_i\neq 0$ for all $i$,
i.e., if the point $[1:0]$ at $\infty$ is not a root of $P$,
then the dehomogenization $p(z)=P(z,1)\in K[z]$ of $P$ is
\[ p(z) = A \prod_{i=1}^m (z-\alpha_i),
\quad\text{where}\quad
\alpha_i = \frac{a_i}{b_i}
\quad\text{and}\quad  A = \prod_{i=1}^d b_i .\]
Similarly, if each $d_i\neq 0$,
then writing $\gamma_i=c_i/d_i$ and $C=\prod_{i=1}^n d_i$,
the dehomogenization $q$ of $Q$ is $q(z) = C\prod_{i=1}^n (z-\gamma_i)$.
Thus, if $b_i,d_i\neq 0$ for all $i$, then
\[ \Res(P,Q) = A^n C^m \prod_{i=1}^m \prod_{j=1}^n (\alpha_i - \gamma_j)
= \Res(p,q) \]
coincides with the classical (nonhomogeneous) resultant, and
\begin{equation}
\label{eq:olddisc}
\Delta(P) = A^{2m-2} \prod_{i<j}(\alpha_i-\alpha_j)^2 = \Delta(p)
\end{equation}
coincides with the classical (nonhomogeneous) discriminant.

For brief expositions on homogeneous discriminants, see
\cite[Section~IX.4]{Lang} or \cite[Section~2.4]{ADS}.
However, the iterated discriminant formulas we need
involve the orbits of critical points.
To incorporate them,
the quotient rule inspires the following definition.
\begin{defin}
\label{def:homdiff}
Let $P,Q\in K[X,Y]$ be relatively prime
homogeneous polynomials of degree $m\geq 1$.
The \emph{homogeneous differential} of $(P,Q)$ is
\[ D_{P,Q} (X,Y) := \frac{1}{Y} (P_X Q - P Q_X)
= \frac{1}{X} (P Q_Y - P_Y Q) \in K[X,Y] ,\]
where $P_X$ denotes the (formal) partial derivative of $P$
with respect to $X$, and similarly for $P_Y, Q_X, Q_Y$.
\end{defin}

A short algebraic computation,
writing $P=A_m X^m + \cdots$ and $Q=B_m X^m+\cdots$,
shows that $D_{P,Q}$ is a homogeneous polynomial of degree $2m-2$ in $K[X,Y]$,
and that the two formulas for it in Definition~\ref{def:homdiff} agree.
In addition, if $p(z),q(z)\in K[z]$ are the dehomogenizations of $P,Q$,
then the derivative of the rational function $p(z)/q(z)$
is $D_{P,Q}(z,1)/(q(z))^2$.
%The homogeneous differential plays the same role as
%the derivative does in the classical resultant formula for
%the discriminant of a polynomial, as shown in the first part of the following result.

\begin{prop}
\label{prop:compres}
Let $P,Q\in K[X,Y]$ be relatively prime
homogeneous polynomials of degree $m\geq 1$,
with homogeneous differential $D=D_{P,Q}$.
Let $\alpha,\beta\in\Kbar$, not both zero,
and define
\[ R(X,Y):=\beta P(X,Y) - \alpha Q(X,Y) . \]
Then
\begin{enumerate}
    \item $\dsps \Res(R,D) = (-1)^{m(m-1)/2} \Res(P,Q) \Delta(R)$.
    \item Let $\tilde{\alpha},\tilde{\beta}\in\Kbar$, not both zero,
and define
\[ \tilde{R}(X,Y):=\tilde{\beta} P(X,Y) - \tilde{\alpha} Q(X,Y) . \]
Then $\dsps \Res(R,\tilde{R})
= (\alpha \tilde{\beta} - \beta\tilde{\alpha})^m \Res(P,Q)$.
\end{enumerate}
\end{prop}

\begin{proof}
Since $R\in \Kbar[X,Y]$ is homogeneous of degree $m$,
there exist points $[\gamma_i,\delta_i]\in\PKbar$ such that
\[ R(X,Y) = \prod_{i=1}^m \big( \delta_i X - \gamma_i Y \big) .\]
Writing $Q=\prod_{i=1}^m (d_i X - c_i Y)$,
observe that equation~\eqref{eq:ResProd} applied to $\Res(R,Q)$ yields
\begin{equation}
    \label{eq:ResRQ}
\prod_{i=1}^m Q(\gamma_i,\delta_i) %= \Res(R,Q)
= (-1)^{m^2} \prod_{i=1}^m R(c_i,d_i)
= (-1)^{m^2} \prod_{i=1}^m \beta P(c_i,d_i) = \beta^m \Res(P,Q) .
\end{equation}
A similar computation applied to $\Res(R,P)$ yields
\begin{equation}
    \label{eq:ResRP}
\prod_{i=1}^m P(\gamma_i,\delta_i) %= \Res(R,P)
%= (-1)^{m^2} \prod_{i=1}^m R(a_i,b_i)
%= (-1)^{m^2} \prod_{i=1}^m (-\alpha) Q(c_i,d_i)
= \alpha^m \Res(P,Q) .
\end{equation}

\textbf{Statement (1)}: We consider two cases.

\medskip

\textbf{Case 1}: If $\beta\neq 0$, then $P=\beta^{-1}(R+\alpha Q)$,
and hence
\begin{equation}
\label{eq:DRx}
D= \frac{1}{\beta Y} \big( R_X Q + \alpha Q_X Q
- R Q_X - \alpha Q Q_X \big) 
= \frac{1}{\beta Y} \big( R_X Q - R Q_X \big) .
\end{equation}
Noting that
\[ R_X = \sum_{i=1}^m \delta_i \prod_{j\neq i}
\big( \delta_j X - \gamma_j Y \big), \]
we have $R_X(\gamma_i,\delta_i)
=\delta_i \prod_{j\neq i} (\gamma_i \delta_j - \delta_i \gamma_j)$
for each $i=1,\ldots,m$, and of course also $R(\gamma_i,\delta_i)=0$.
Thus, equation~\eqref{eq:DRx} yields
\[ D(\gamma_i,\delta_i) = \beta^{-1} Q(\gamma_i , \delta_i)
\prod_{j\neq i} \big(\gamma_i \delta_j - \delta_i \gamma_j\big), \]
for each $i=1,\ldots, m$. Therefore, by equation~\eqref{eq:ResProd}, we have
\begin{align*}
\Res(R,D) &= \prod_{i=1}^m D(\gamma_i,\delta_i)
= \bigg( \beta^{-m} \prod_{i=1}^m Q(\gamma_i, \delta_i) \bigg)
\prod_{i=1}^m \prod_{j\neq i}
\big(\gamma_i \delta_j - \delta_i \gamma_j\big)
\\
&= (-1)^{m(m-1)/2} \bigg( \beta^{-m} \prod_{i=1}^m Q(\gamma_i, \delta_i) \bigg)
\prod_{1\leq i < j \leq m} \big(\gamma_i \delta_j - \delta_i \gamma_j\big)^2,
\\
&= (-1)^{m(m-1)/2} \Res(P,Q) \Delta(R),
\end{align*}
by rearranging the final product in the first line
and then applying equation~\eqref{eq:ResRQ}.

\medskip

\textbf{Case 2}: If $\alpha\neq 0$, then $Q=\alpha^{-1}(\beta P - R)$.
A similar argument as in Case~1 yields
\[ D= \frac{1}{\alpha X} \big( P_Y R - PR_Y \big)
\quad\text{and}\quad R_Y(\gamma_i , \delta_i)
= -\gamma_i \prod_{j\neq i} \big(\gamma_i \delta_j - \delta_i \gamma_j\big).\]
Thus,
\begin{align*}
\Res(R,D) &= \prod_{i=1}^m D(\gamma_i,\delta_i)
= \bigg( \alpha^{-m} \prod_{i=1}^m P(\gamma_i, \delta_i) \bigg)
\prod_{i=1}^m \prod_{j\neq i}
\big(\gamma_i \delta_j - \delta_i \gamma_j\big)
\\
&= (-1)^{m(m-1)/2} \Res(P,Q) \Delta(R),
\end{align*}
by similar reasoning as in Case~1, using equation~\eqref{eq:ResRP}.

\medskip

\textbf{Statement (2)}:
Without loss, we may assume $\beta\neq 0$. Then
\begin{align*}
\Res(R,\tilde{R}) &=
\prod_{i=1}^m \tilde{R}(\gamma_i, \delta_i)
= \beta^{-m} \prod_{i=1}^m
\big(\beta\tilde{\beta} P(\gamma_i, \delta_i)
- \beta\tilde{\alpha} Q(\gamma_i, \delta_i) \big)
\\
&= \beta^{-m} \prod_{i=1}^m
\big(\alpha\tilde{\beta} Q(\gamma_i, \delta_i)
- \beta\tilde{\alpha} Q(\gamma_i, \delta_i) \big)
\\
&= \beta^{-m} (\alpha \tilde{\beta} - \beta\tilde{\alpha})^m
\prod_{i=1}^m Q(\gamma_i, \delta_i)
= (\alpha \tilde{\beta} - \beta\tilde{\alpha})^m \Res(P,Q),
\end{align*}
where the first equality is by equation~\eqref{eq:ResProd},
the third is because $R(\gamma_i,\delta_i)=0$,
and the fifth is by equation~\eqref{eq:ResRQ}.
\end{proof}

\begin{thm}
\label{thm:compdisc}
Let $P,Q\in K[X,Y]$ be relatively prime
homogeneous polynomials of degree $m\geq 1$,
with homogeneous differential $D=D_{P,Q}$.
Let $J\in K[X,Y]$ be a homogeneous polynomial of degree $n\geq 1$,
and let
\[ H(X,Y):=J(P(X,Y),Q(X,Y)). \]
Then
\[ \Delta(H) = (-1)^{mn(m-1)/2} \Delta(J)^m \Res(P,Q)^{n(n-2)} \Res(H,D) .\]
\end{thm}

\begin{proof}
Write $J(X,Y)=\prod_{i=1}^n (\beta_i X - \alpha_i Y)$,
with $[\alpha_i,\beta_i]\in \PKbar$. Then
\[ H(X,Y) = \prod_{i=1}^n R_i(X,Y)
\quad\text{where}\quad
R_i = \beta_i P - \alpha_i Q .\]
Writing $R_i(X,Y)=\prod_{j=1}^m (\delta_{ij} X - \gamma_{ij} Y)$,
we have
\begin{align}
\label{eq:DeltaH}
\Delta(H) &= \Bigg[ \prod_{i=1}^n \prod_{1\leq j < k \leq m}
\big(\gamma_{ij} \delta_{ik} - \gamma_{ik} \delta_{ij} \big)^2 \Bigg]
\Bigg[ \prod_{1\leq i < j \leq n} \prod_{k=1}^n \prod_{\ell=1}^n
\big(\gamma_{ik} \delta_{j\ell} - \gamma_{j\ell} \delta_{ik} \big)^2 \Bigg]
\notag \\
&= \Bigg[ \prod_{i=1}^n \Delta(R_i) \Bigg]
\Bigg[ \prod_{1\leq i < j \leq n} \Res(R_i,R_j)^2 \Bigg] .
\end{align}
The first product in equation~\eqref{eq:DeltaH} is
\begin{align*}
%\label{eq:DeltaH1}
\prod_{i=1}^n \Delta(R_i) &=
\prod_{i=1}^n \Big[ (-1)^{m(m-1)/2} \Res(P,Q)^{-1} \Res(R_i,D) \Big]
\\
&=(-1)^{mn(m-1)/2} \Res(P,Q)^{-n} \Res(H,D) ,
\end{align*}
where the first equality is by Proposition~\ref{prop:compres}(1),
and the second is by equation~\eqref{eq:ResProd}
and the fact that $H=\prod_{i=1}^n R_i$.
The second product in equation~\eqref{eq:DeltaH} is
\begin{align*}
\prod_{1\leq i < j \leq n} \Res(R_i,R_j)^2
&= \prod_{1\leq i < j \leq n}
\Big[ (\alpha_i \beta_j - \alpha_j \beta_i)^{2m} \Res(P,Q)^2 \Big]
\\
&= \Res(P,Q)^{n(n-1)} \Bigg( \prod_{1\leq i < j \leq n}
(\alpha_i \beta_j - \alpha_j \beta_i)^2 \bigg)^m
\\
&= \Res(P,Q)^{n(n-1)} \Delta(J)^m,
\end{align*}
where the first equality is by Proposition~\ref{prop:compres}(2).
Thus, equation~\eqref{eq:DeltaH} yields the desired formula.
\end{proof}

A rational function $f(z)=p(z)/q(z)\in K(z)$ of degree $d\geq 1$
may be written in homogeneous coordinates as $F(X,Y)=(P(X,Y),Q(X,Y))$,
where
\[ P(X,Y):=Y^d p(X/Y), \quad\text{and}\quad Q(X,Y)=Y^d q(X,Y), \]
which are both homogeneous polynomials of degree $d$.
Of course, this lift to homogeneous coordinates is not unique,
as we may multiply both $P$ and $Q$ by the same constant
$\lambda\in K^{\times}$. Thus, having fixed a choice of $P$ and $Q$,
the iterate $F^n=(P_n,Q_n)$ is a choice of lift of $f^n$ to homogeneous
coordinates.

\begin{cor}
\label{cor:iterdisc}
Let $P,Q\in K[X,Y]$ be relatively prime
homogeneous polynomials of degree $d\geq 1$,
with homogeneous differential $D=D_{P,Q}$.
Write
\[ D(X,Y)= c \prod_{i=1}^{2d-2} (\theta_i X - \eta_i Y)
\quad \text{with} \quad c\in K^{\times} \text{ and } [\eta_i,\theta_i]\in\PKbar . \]
Let $F=(P,Q)$, and let $s_0,t_0\in K$ not both zero. For each $n\geq 0$, define
\[ H_n(X,Y) = t_0 P_n(X,Y) - s_0 Q_n(X,Y) \in K[X,Y],\]
where $(P_n,Q_n) = F^n \in K[X,Y]\times K[X,Y]$.
Then for each $n\geq 1$, we have
\begin{equation}
\label{eq:iterdisc}
\Delta(H_n) =
(-1)^{d^n(d-1)/2} c^{d^n} \Delta(H_{n-1})^d \Res(P,Q)^{d^{n-1}(d^{n-1}-2)}
\prod_{i=1}^{2d-2} H_n(\eta_i,\theta_i) .
\end{equation}
\end{cor}

\begin{proof}
Note that $H_0(X,Y)=t_0 X- s_0 Y$ is homogeneous of degree $1$,
and for $n\geq 1$, we have that $H_n=H_{n-1}\circ F$
is homogeneous of degree $d^n$.
Pulling the constant $c$ out of $D$
(and raising it to the power $\deg(H_n)=d^n$)
and applying equation~\eqref{eq:ResProd},
we have
\[ \Res(H_n,D) = (-1)^{d^n(2d-2)} c^{d^n} \prod_{i=1}^{2d-2} H_n(\eta_i,\theta_i)
= c^{d^n} \prod_{i=1}^{2d-2} H_n(\eta_i,\theta_i) .\]
Thus, with $H=H_{n}$, $J=H_{n-1}$, $m=d$, and using $d^{n-1}$ in the role of $n$,
the desired formula is immediate from Theorem~\ref{thm:compdisc}.
\end{proof}

\begin{thm}
\label{thm:discsquare}
With notation as in Corollary~\ref{cor:iterdisc},
suppose the degree is $d=2$, and that $\charact K \neq 2$.
%Let $\delta':=\Res(P,Q)\in K^{\times}$.
Suppose that there is an integer $\ell\geq 1$ such that
\[ F^{\ell}([\eta_1,\theta_1]) = F^{\ell}([\eta_2,\theta_2]) \text{ as points in } \PKbar, \]
but
\[ F^{\ell-1}([\eta_1,\theta_1]) \neq F^{\ell-1}([\eta_2,\theta_2]) \text{ as points in } \PKbar. \]
Then $\ell\geq 2$, and we have
%\[ \delta' \prod_{i=1}^2 H_{\ell}(\eta_i,\theta_i) \in K^2
\[ \Res(P,Q) \prod_{i=1}^2 H_{\ell}(\eta_i,\theta_i) \in K^2
\quad \text{and} \quad \prod_{i=1}^2 H_{n}(\eta_i,\theta_i) \in K^2
\text{ for every } n\geq \ell+1,\]
where $K^2$ denotes the set of squares of elements of $K$.
\end{thm}

\begin{proof}
The two critical points $[\eta_1,\theta_1]$ and $[\eta_2,\theta_2]$
of the morphism $F:\PP^1\to\PP^1$ given by $F=[P,Q]$
must be distinct, since $\deg(F)=2$, and a higher-multiplicity critical point
would result in a strictly larger local degree.
(Recall that we have $\charact K\neq 2$.)
In addition, if $F([\eta_1,\theta_1])=F([\eta_2,\theta_2])$, then this common
point would have at least four preimages (counting multiplicity),
again contradicting the fact that $\deg(F)=2$.
Thus, the smallest iterate $\ell$ for which 
$F^{\ell}([\eta_1,\theta_1]) = F^{\ell}([\eta_2,\theta_2])$
must satisfy $\ell\geq 2$, as claimed.

For any $\lambda\in\Kbar^{\times}$,
replacing $(\eta_1,\theta_1)$ by $(\lambda\eta_1,\lambda\theta_1)$
and $(\eta_2,\theta_2)$ by $(\lambda^{-1}\eta_2,\lambda^{-1}\theta_2)$
changes neither the points $F^{n}([\eta_i,\theta_i])$
nor the product $\prod_{i=1}^2 H_{n}(\eta_i,\theta_i)$.
In addition, for any $c\in K^{\times}$,
replacing $(\eta_1,\theta_1)$ by $(c\eta_1,c\theta_1)$
without changing $(\eta_2,\theta_2)$
similarly does not change the points $F^{n}([\eta_i,\theta_i])$,
but it changes the product $\prod_{i=1}^2 H_{n}(\eta_i,\theta_i)$
by a factor of $c^{2^n}$, which is a square in $K$ for $n\geq 1$.
Thus, since the quadratic form $D_{P,Q}(X,Y)$ is defined over $K$,
we may assume that $\eta_1,\eta_2,\theta_1,\theta_2\in L$,
where $L$ is either $K$ or a quadratic extension of $K$;
and in the latter case, we may further assume that
$(\eta_2,\theta_2)$ is $\Gal(L/K)$-conjugate to $(\eta_1,\theta_1)$.

Let $[a,b]\in\PKbar$ be the point
\[ [a,b] := F^{\ell}([\eta_1,\theta_1]) = F^{\ell}([\eta_2,\theta_2]) .\]
Since this point is either defined over $K$ already,
or else defined over $L$ and $\Gal(L/K)$-conjugate to itself,
we may assume that $a,b\in K$.

For each $i=1,2$, define $\alpha_i,\beta_i\in L$ by
\[ (\alpha_i,\beta_i) :=F^{\ell-1}(\eta_i,\theta_i) .\]
Then since the two points $[\alpha_i,\beta_i]\in\PKbar$ are distinct,
the quadratic form $bP-aQ\in K[X,Y]$ must factor as
\[ bP(X,Y) - aQ(X,Y) = \mu (\beta_1 X - \alpha_1 Y)(\beta_2 X - \alpha_2 Y) \]
for some $\mu\in K^{\times}$. Thus,
\begin{align*}
\prod_{i=1}^2 H_{\ell}(\eta_i,\theta_i)
& = \prod_{i=1}^2 \big( t_0 P(\alpha_i,\beta_i) - s_0 Q(\alpha_i,\beta_i) \big)
= \Res\big( t_0 P - s_0 Q, \mu^{-1} (bP-aQ) \big)
\\
&= \mu^{-2} \Res(t_0 P - s_0 Q, bP-aQ)
= \mu^{-2} (s_0 b - t_0a)^2 \Res(P,Q),
\end{align*}
where the last equality is by 
Proposition~\ref{prop:compres}(2).
%Multiplying by $\delta'=\Res(P,Q)\in K$
Multiplying by $\Res(P,Q)\in K^{\times}$
yields that the first desired product is indeed a square in $K$.

For the second product, i.e., for $n\geq \ell+1$, first consider the
case that $\eta_1,\eta_2,\theta_1,\theta_2\in K$.
Then $F^{n-1}(\eta_1,\theta_1)$ and $F^{n-1}(\eta_2,\theta_2)$
are both in $K\times K$, and moreover they describe the same point in $\PK$,
since $n-1\geq \ell$.
Hence, there is some $\mu\in K^\times$ such that
\[ F^{n-1}(\eta_2,\theta_2) = \mu F^{n-1}(\eta_1,\theta_1),
\quad\text{and hence}\quad
F^{n}(\eta_2,\theta_2) = \mu^2 F^{n}(\eta_1,\theta_1). \]
It follows that
\[ \prod_{i=1}^2 H_{n}(\eta_i,\theta_i)
= \mu^2 \big(H_{n}(\eta_1,\theta_1)\big)^2 \in K^2, \]
as desired.

The other case is that 
$(\eta_2,\theta_2)$ is $\Gal(L/K)$-conjugate to $(\eta_1,\theta_1)$,
where $L$ is a quadratic extension of $K$.
Then $F^{n-1}(\eta_1,\theta_1)$ and $F^{n-1}(\eta_2,\theta_2)$
are both in $L\times L$ and are also Galois conjugate.
As in the previous case, they also describe the same point in $\PKbar$,
so this point must be $K$-rational. Hence, there is some
$\mu\in L^{\times}$ and some $a_{n-1},b_{n-1}\in K$ not both zero
such that
\[ F^{n-1}(\eta_1,\theta_1)=\mu (a_{n-1},b_{n-1})
\quad \text{and} \quad
F^{n-1}(\eta_2,\theta_2)=\sigma(\mu) (a_{n-1},b_{n-1}) , \]
where $\sigma$ is the nontrivial element of $\Gal(L/K)$.
Let $\gamma=\mu\sigma(\mu)\in K^{\times}$. Then
\[ \prod_{i=1}^2 H_{n}(\eta_i,\theta_i)
= \gamma^2 \big(t_0 P(a_{n-1},b_{n-1}) - s_0 Q(a_{n-1},b_{n-1} ) \big)^2
\in K^2. \qedhere \]
\end{proof}

Still in the case $d=2$, we also have the following identity.

\begin{prop}
\label{prop:DPQformula}
Let $P,Q\in K[X,Y]$ be relatively prime homogeneous polynomials
of degree~2, with homogeneous differential $D_{P,Q}$.
Then $\Delta(D_{P,Q})=4\Res(P,Q)$.
\end{prop}

\begin{proof}
This is a brute-force calculation. Writing
\[ P=a_0 X^2 + a_1 XY + a_2 Y^2 \quad\text{and}\quad
Q=b_0 X^2 + b_1 XY + b_2 Y^2 , \]
direct computation
shows that both sides of the desired identity are equal to
\[ 4(a_2 b_0 - a_0 b_2)^2 - 4(a_2 b_1 - a_1 b_2)(a_1 b_0 - a_0 b_1).
\qedhere \]
\end{proof}

\begin{remark}
\label{rem:poly}
Corollary~\ref{cor:iterdisc} also yields the following
formula for the (nonhomogeneous) discriminants 
of iterated polynomials.
Let $f(z)\in K[z]$ be a polynomial of degree $d\geq 2$
with lead coefficient $A\in K^{\times}$, and let $x_0\in K$.
Then for every $n\geq 1$, we have
\begin{equation}
\label{eq:polyiter}
\Delta\big(f^n- x_0 \big)
= (-1)^{d^n (d-1)/2} d^{d^n} A^{d^{2n-1}-1}
\big( \Delta(f^{n-1}-x_0)\big)^d
\prod_{f'(c)=0} \big( f^n(c)-x_0 \big),
\end{equation}
where the product is over all finite critical points of $f$,
repeated according to multiplicity.

Indeed, lifting $f$ to homogeneous coordinates as
$F[X,Y] = (P,Q)=(Y^d f(X/Y),Y^d)$, simple computations yield
\[ \Res(P,Q) = A^d \quad\text{and}\quad
D=D_{P,Q}=dA Y^{d-1} \prod_{f'(c)=0} (X-cY).\]
Writing $H_n=P_n -x_0 Y^{d^n}$, where $P_n=Y^{d^n} f^n(X/Y)$,
we have $\Delta(H_n) = \Delta(f^n-x_0)$
by equation~\eqref{eq:olddisc}, and it is easy to check that the
$\prod H_n(\eta_i,\theta_i)$ term in Corollary~\ref{cor:iterdisc}
becomes
\[ d^{d^n} A^{2d^n-1} \prod_{f'(c)=0} \big( f^n(c)-x_0 \big) .\]
The desired equation is then immediate from Corollary~\ref{cor:iterdisc}.

A variant of formula~\eqref{eq:polyiter} appeared in
\cite[Proposition~3.2]{AHM}, and a version similar to~\eqref{eq:polyiter}
appeared with an incorrect power of the lead coefficient
in \cite[Equation~(1)]{BenJuu}.
\end{remark}

\section{Colliding critical points and $M_{\infty}$}
\label{sec:collideM}
We are now prepared to prove our first main result, Theorem~\ref{thm:main1}.
The central goal of the proof is to produce a labeling of the tree of preimages
with respect to which every $\sigma\in G_{\infty}$
acts as an element of $\widetilde{M}_{\ell}$ or $M_{\ell}$.
After some preliminaries, our strategy will be to start
with a completely arbitrary labeling of the tree,
and then to make successive
changes to the labeling until it has this desired property.

\begin{proof}[Proof of Theorem~\ref{thm:main1}]
%\label{prf: main1}
Writing $f=p/q$ where $p,q\in K[z]$ with $\max\{\deg p, \deg q\}=2$,
define $P,Q\in K[X,Y]$ by 
\[P(X,Y):=Y^2 p(X/Y)  \quad \text{ and } \quad Q(X,Y):=Y^2 q(X/Y).\]
Thus, $F:=(P,Q)$ is a homogenization of $f$,
and the homogeneous differential $D=D_{P,Q}$ may be factored as
\[D(X,Y)=c (\theta_1 X - \eta_1 Y)(\theta_2 X - \eta_2 Y),\]
with $c\in K^{\times}$ and $\xi_i=\eta_i/\theta_i$ for $i=1,2$.
(Here, we understand the latter expression to be the point at $\infty$ if $\theta_i=0$.)

%Let $\delta'=\Res(P,Q)\in K^{\times}$.
Any point $x$ in the backward orbit $\Orb_f^-(x_0)$ corresponds to a node
of the tree $T_{\infty}$, and we also call this node $x$. Writing
$x=[s,t]\in\PKbar$ with $s,t\in K(x)$, define
\[H_{x,\ell}:=t P_\ell - s Q_\ell\in K(x)[X,Y],\]
where $P_\ell$ and $Q_\ell$ are the coordinate functions of $F^\ell$,
i.e., where $F^{\ell} = (P_\ell,Q_\ell)$.

Choose any labeling of the tree $T_{\infty}$.
In the rest of the proof, we will make successive changes to
this labeling until it successfully exhibits $G_{\infty}$ as a subgroup
of $\widetilde{M}_{\ell}$ or $M_{\ell}$.

\smallskip

\textbf{Case 1}. If $\delta$ is a square in $K$,
then $\xi_1,\xi_2\in\PK$. Thus, we may assume that
$\eta_1,\eta_2,\theta_1,\theta_2\in K$.
The discriminant $\Delta(D_{P,Q})$ is therefore also a square in $K$,
and hence, by Proposition~\ref{prop:DPQformula}, so is
$\Res(P,Q)=4\Delta(D_{P,Q})$.
%$\delta':=\Res(P,Q)=4\delta$ is also a square in $K$.
%$\delta=\Delta(D_{P,Q})/4 = c^2(\eta_1\theta_2 - \eta_2\theta_1)^2/4$ is a square in $K$.

With $d=2$ in equation~\eqref{eq:iterdisc} of Corollary~\ref{cor:iterdisc},
it follows that for any $x\in \Orb_f^-(x_0)$,
the discriminant $\Delta(H_{x,\ell})$ is a square in $K(x)$.
Then by Theorem~\ref{thm:discsquare},
using the fact that $\Res(P,Q) \in K^2$, we also have
%using the fact that $\delta'=\Res(P,Q) \in K^2$, we also have
\[ \prod_{i=1}^2 H_{x,\ell}(\eta_i,\theta_i)\in K(x)^2 . \]

We may also write
\[ H_{x,\ell} = \prod_{i=1}^{2^\ell} (\beta_{\ell,i} X - \alpha_{\ell,i} Y), \]
where $\{ [\alpha_{\ell,i},\beta_{\ell,i}] \, | \, i=1,\ldots, 2^{\ell} \}$ are the preimages
in $\PKbar$ of $x$ under $f^\ell$.
For any $\sigma\in G_{\infty}$, it follows that
\begin{equation}
\label{eq:sigmaeven}
\text{if } \sigma(x)=x, \quad\text{then}\quad \sgn_\ell(\sigma,x)=+1,
\end{equation}
because any such $\sigma$ permutes these $2^{\ell}$ preimages,
and it does so with even parity,
since $\Delta(H_{x,\ell})$ is a square in $K(x)$.

In particular, implication~\eqref{eq:sigmaeven} shows that $\sgn_\ell(\sigma,x_0)=+1$
for every $\sigma\in G_{\infty}$, since each such $\sigma$ fixes $x_0\in\PK$.
We will now proceed inductively up the tree, making adjustments to the labeling
as we go. For any $m\geq 1$,
suppose that we have already verified that for all nodes $x$ up to level $m-1$,
we have $\sgn_\ell(\sigma,x)=+1$ for all $\sigma\in G_{\infty}$.

Given a node $y$ at level $m$, its Galois orbit $G_{\infty}(y)$
consists of nodes at the same level $m$ of the tree.
For each such node $w\in G_{\infty}(y)$, choose $\sigma_w\in G_{\infty}$
such that $\sigma_w(y)=w$. Note that if $\sigma'_w\in G_{\infty}$
also satisfies $\sigma'_w(y)=w$, then
\[ \sgn_{\ell}(\sigma'_w,y) = \sgn_{\ell}(\sigma_w,y),
\quad\text{since} \quad
\sgn_{\ell}(\sigma_w^{-1}\sigma'_w, y)=+1 \]
by implication~\eqref{eq:sigmaeven}. Define
\[ W_y := \{ w\in G_{\infty}(y) \, | \, \sgn_\ell(\sigma_w,y)= -1 \}. \]
Observe that $y\not\in W_y$, again by implication~\eqref{eq:sigmaeven}.

We now modify the labeling.
For each node $w\in W_y$, transpose the labels of two nodes that lie $\ell$ levels
above $w$ (and which share the same parent $\ell-1$ levels above $w$).
Since we made no change to
the labeling above $y$ but made this single transposition above $w$,
the new labeling now gives $\sgn_\ell(\sigma_w,y)=+1$.

Moreover, for any two nodes $w$ and $z$ in the Galois orbit $G_{\infty}(y)$,
and for any $\rho\in G_{\infty}$ with $\rho(w)=z$,
we claim that $\sgn_{\ell}(\rho,w)=+1$ under this new labeling.
Indeed, the Galois automorphism
\[ \lambda:=\sigma_z^{-1}\circ\rho\circ\sigma_w \]
fixes $y$
and hence, by implication~\eqref{eq:sigmaeven},
must have sign $\sgn_\ell(\lambda,y)=+1$.
In addition, by the previous paragraph, the signs
$\sgn_{\ell}(\sigma_w,y)$ and $\sgn_{\ell}(\sigma_z,y)$ are also both $+1$.
Hence, $\sgn_{\ell}(\sigma_w^{-1},w)=+1$ as well. Therefore, by
equation~\eqref{eq:Parident}, we have
\begin{align*}
\sgn_{\ell}(\rho,w) &= \sgn_{\ell}(\sigma_z \lambda \sigma_w^{-1}, w)
=\sgn_{\ell}(\sigma_z,y) \cdot \sgn_{\ell}(\lambda, y) \cdot \sgn_{\ell}(\sigma_w^{-1},w)
\\
& = (+1) \cdot(+1) \cdot(+1) = +1,
\end{align*}
proving our claim that $\sgn_{\ell}(\rho,w)=+1$.

Repeat this relabeling for each of the (finitely many) Galois orbits
among the $2^m$ nodes at level $m$. That is, for each such orbit,
choose a node $y$ in the orbit, define the set $W_y$ as above,
and adjust the labels above each $w\in W_y$. Having completed this process
for each Galois orbit at level $m$, it follows that for any node $x$
at level $m$ and any $\sigma\in G_{\infty}$, we have
$\sgn_{\ell}(\sigma,x)=+1$.
Hence, extending this relabeling inductively up the tree,
we have $G_{\infty}\subseteq M_{\ell}$,
proving the reverse implication of part~(2) of the theorem.

\smallskip

\textbf{Case 2}. If $\delta$ is not a square in $K$, then
the critical points $\xi_1,\xi_2$
are Galois conjugate and defined over the quadratic extension
$L:= K(\sqrt{\delta})$ of $K$.
In addition, $\Res(P,Q)$ is not a square in $K$, since
as in Case~1, it is a square times $\delta$.
%By Proposition~\ref{prop:DPQformula}, this extension must be 

Defining $H_{x,\ell}$ as in Case~1, observe that Corollary~\ref{cor:iterdisc}
and Theorem~\ref{thm:discsquare} together show that
for any $x\in\PKbar$, the quantity $\delta \Delta(H_{x,\ell})$ is a square in $K(x)$.
Since $\sqrt{\Delta(H_{x,\ell})}$ is an arithmetic combination of the points
in $f^{-\ell}(x)$ (see Definition~\ref{def:hdisres} and equation~\eqref{eq:olddisc}),
it follows that $\sqrt{\delta}\in K(x)_{\ell}$,
where $K(x)_{\ell}$ is the extension of $K(x)$ obtained by adjoining $f^{-\ell}(x)$.
In particular,
\[ \sqrt{\delta}\in K(x_0)_{\ell}=K_{\ell}\subseteq K_{\infty} .\]
Therefore, it makes sense to define $G'_{\infty}:=\Gal(K_{\infty}/L)$,
which is a subgroup of $G_{\infty}$ of index~2.

Fix $\tau\in G_{\infty}\smallsetminus G'_{\infty}$.
Note that the two cosets
of $G'_{\infty}$ in $G_{\infty}$ are
\[ G'_{\infty} = \{\sigma\in G_{\infty} \, | \, \sigma(\sqrt{\delta}) = \sqrt{\delta} \}
\quad\text{and}\quad
G'_{\infty}\tau = \{\sigma\in G_{\infty} \, | \, \sigma(\sqrt{\delta}) = -\sqrt{\delta} \}. \]

Applying Case~1 with $L$ in place of $K$ and $G'_{\infty}$ in place of $G_{\infty}$,
we may label the tree so that $G'_{\infty}\subseteq M_{\ell}$.
Even after this relabeling, note that $\sgn_{\ell}(\tau,x_0)=-1$,
since, as noted above, $\Delta(H_{x_0,\ell})$ is $\delta$ times a square in $K(x_0)=K$,
and hence $\tau$ must map this discriminant to its negative, meaning that
it acts as an odd permutation on the $2^{\ell}$ points of $f^{-\ell}(x_0)$.
In particular, we have $\tau\in G_{\infty}\smallsetminus M_{\ell}$,
proving the forward implication of part~(2) of the theorem.

We will now make some further adjustments to this labeling.
For any $m\geq 1$,
suppose that we have already verified that for all nodes $x$ up to level $m-1$,
we have
\[ \sgn_\ell(\sigma,x)=\sgn_\ell(\sigma,x_0) \quad\text{for all}\quad \sigma\in G_{\infty}, \]
a condition which holds vacuously at level $0$.
Given a node $y$ at level $m$, the Galois orbit $G'_{\infty}(y)$
either coincides with $G_{\infty}(y)$ or is a subset of exactly half the size,
since $[G_{\infty}:G'_{\infty}]=2$.

In the former case, for every node $w$ in the same orbit $G'_{\infty}(y)=G_{\infty}(y)$,
there is some $\rho_w\in G'_{\infty}$ such that $\rho_w(w)=\tau(w)$.
Note that $\rho_w^{-1}\tau$ maps $\sqrt{\delta}$ to its negative and hence also maps
$\sqrt{\Delta(H_{w,\ell})}$ to its negative, as once again,
$\Delta(H_{w,\ell})$ is $\delta$ times a square in $K(w)$.
Since $\rho_w^{-1}\tau$ fixes $w$, it follows that
$\sgn_{\ell}(\rho_w^{-1}\tau,w)= -1$. Thus, by equation~\eqref{eq:Parident},
for any $\sigma\in G'_{\infty}$, we have
\begin{align*}
\sgn_{\ell}(\sigma\tau,w) &=
\sgn_{\ell}(\sigma,\tau(w)) \cdot \sgn_{\ell}(\rho_w, w) \cdot \sgn_{\ell}(\rho_w^{-1}\tau,w)
= (+1) \cdot (+1) \cdot (-1)
\\
& = (+1) \cdot (-1)
= \sgn_{\ell}(\sigma,x_0) \cdot \sgn_{\ell}(\tau,x_0) = \sgn_{\ell}(\sigma\tau,x_0).
\end{align*}
Hence, for every $\sigma$ in either of the two cosets of $G'_{\infty}$ in $G_{\infty}$,
we have $\sgn_{\ell}(\sigma,w) =\sgn_{\ell}(\sigma,x_0)$, as desired.

In the latter case, the Galois orbit $G_{\infty}(y)$ is the disjoint union
of $V_{y,0}:=G'_{\infty}(y)$ and $V_{y,1}:=G'_{\infty}\tau(y)$.
If $\sgn_{\ell}(\tau,y)=+1$, then for each node $w\in V_{y,1}$,
transpose the labels of two nodes that lie $\ell$ levels above $w$
and share the same parent, as we did in Case~1.
Since every $\sigma\in G'_{\infty}$ maps the set of nodes $V_{y,1}$ to itself,
this change in labeling preserves the inclusion $G'_{\infty}\subseteq M_{\ell}$,
but now we may assume that $\sgn_{\ell}(\tau,y)=-1$.

Therefore, for any $\sigma_1,\sigma_2\in G'_{\infty}\subseteq M_{\ell}$,
equation~\eqref{eq:Parident} yields
\begin{align}
\label{eq:oddsum}
\sgn_{\ell}(\sigma_1 \tau \sigma_2^{-1}, \sigma_2(y))
& = \sgn_{\ell}(\sigma_1, \tau(y)) \cdot
\sgn_{\ell}(\tau,y) \cdot \sgn_{\ell}(\sigma_2^{-1}, \sigma_2(y))
\\
&= (+1) \cdot (-1) \cdot (+1) = -1 . \notag
\end{align}
Given any $w\in V_{y,0}$ and any $\rho\in G'_{\infty}\tau$,
we may write $w=\sigma_2(y)$ and $\rho=\sigma_1\tau\sigma_2^{-1}$
for some $\sigma_1,\sigma_2\in G'_{\infty}$.
Hence, equation~\eqref{eq:oddsum} becomes
\begin{equation}
\label{eq:oddsum2}
\sgn_{\ell}(\rho,w) = -1 = \sgn_{\ell}(\rho,x_0).
\end{equation}
In addition, given any $w\in V_{y,1}$ and any $\rho\in G'_{\infty}\tau$,
then applying equation~\eqref{eq:oddsum2} to $\rho^{-1}$, with $\rho(w)$ in the role of $w$,
gives us
\[ \sgn_{\ell}(\rho^{-1},\rho(w)) = \sgn_{\ell}(\rho^{-1},x_0) . \]
Taking inverses once again gives $\sgn_{\ell}(\rho,w) = \sgn_{\ell}(\rho,x_0)$;
so we have shown that this equality holds for every
$w$ in the Galois orbit $G_{\infty}(y)$ and any $\rho\in G_{\infty}$.

As in Case~1, repeat this relabeling for each of the Galois orbits
of nodes at level $m$ of the tree. Having completed this process at level $m$,
it follows that for any node $x$ at level $m$ and any $\sigma\in G_{\infty}$,
we have $\sgn(\sigma,x) = \sgn(\sigma,x_0)$.
Hence, extending this relabeling inductively up the tree,
we have $G_{\infty}\subseteq \widetilde{M}_{\ell}$,
proving statement~(1).
\end{proof}

\begin{remark}
Theorem~\ref{thm:main1} is essentially the content of
Theorem~4.9.3 of \cite{Pink2}. We have provided the proof above
both to present a different argument and to illustrate that the result
can be proven directly from our parity- and discriminant-based
descriptions of the groups $M_{\ell}$ and $\widetilde{M}_{\ell}$.
\end{remark}

We close this section with a result describing the action of the finite Galois group $G_n$
on the finite tree $T_n$, and its relationship to the groups
$M_{\ell,n}$ and $\widetilde{M}_{\ell,n}$ of Definition~\ref{def:Mfinite}.
The hypotheses are exactly the same as in Theorem~\ref{thm:main1}.

\begin{cor}
\label{cor:main1}
Let $K$ be a field of characteristic different from $2$, and
let $f\in K(z)$ be a rational function of degree $2$ with critical points
$\xi_1,\xi_2\in\PKbar$.
Let $\delta\in K^{\times}$ be the discriminant of the minimal polynomial
of $\xi_1$ over $K$,
which we understand to be $\delta=1$ if $\xi_1\in\PK$.
Fix $x_0\in\PK$, and let $G_{\infty}$ be the arboreal Galois group for
$f$ over $K$, rooted at $x_0$.
%Suppose that $f^{\ell}(\xi_1)=f^{\ell}(\xi_2)$ and $f^{\ell-1}(\xi_1)\neq f^{\ell-1}(\xi_2)$
%for some integer $\ell\geq 2$.
Suppose that $\xi_1$ and $\xi_2$ collide at the $\ell$-th iterate under $f$,
for some integer $\ell\geq 2$. Let $n\geq 0$ be an integer.
\begin{enumerate}
\item $G_{n}$ is isomorphic to a subgroup of $\widetilde{M}_{\ell,n}$,
via an appropriate labeling of the tree.
\item If $n\geq\ell$, then $G_{n}$ is isomorphic to a subgroup of $M_{\ell,n}$
if and only if $\delta$ is a square in $K$.
\end{enumerate}
\end{cor}

\begin{proof}
Recall that throughout this paper,
we understand an isomorphism between two groups that act on a tree to be
an equivariant isomorphism, as described in the introduction.
Therefore, statement~(1) of Corollary~\ref{cor:main1} is immediate from
statement~(1) of Theorem~\ref{thm:main1}.

For statement~(2), since we have assumed $n\geq \ell$, we have
that $G_n$ is a subgroup of $M_{\ell,n}$ if and only if
every $\overline{\sigma}\in G_n$ satisfies $\sgn_{\ell}(\overline{\sigma},x_0)=+1$.
By the definition $\widetilde{M}_{\ell,n}$ via restriction of elements of $\widetilde{M}_{\ell}$,
this latter condition holds if and only if 
every $\sigma\in G_\infty$ satisfies $\sgn_{\ell}(\sigma,x_0)=+1$.
Thus, $G_n$ is a subgroup of $M_{\ell,n}$ if and only if
$G_\infty$ is a subgroup of $M_{\ell}$;
by statement~(2) of Theorem~\ref{thm:main1},
this occurs if and only if $\delta$ is a square in $K$.
\end{proof}

\section{Odd cousins}
\label{sec:cousins}

In this section, we investigate certain aspects of the group $M_{\ell}$.

\begin{defin}
\label{def:cousins}
Fix a labeling on the tree $T_{\infty}$ and an integer $\ell\geq 2$,
and let $M_{\ell}$ be the associated subgroup of $\Aut(T_{\infty})$,
as in Definition~\ref{def:PinkGroup}.
Let $\sigma\in M_{\ell}$.
\begin{enumerate}
\item Let $y_0,y_1$ be the two children of a node $w$,
i.e., the two nodes connected to $w$
on the level above $w$.
We say that $\sigma$
\[ \text{\emph{acts $\ell$-positively above $w$} if }
\sgn_{\ell -1}(\sigma,y_0) = \sgn_{\ell -1}(\sigma,y_1) = +1, \]
or that $\sigma$
\[ \text{\emph{acts $\ell$-negatively above $w$} if }
\sgn_{\ell -1}(\sigma,y_0) = \sgn_{\ell -1}(\sigma,y_1) = -1. \]
\item
Let $n\geq \ell$, let $m:=2^{n-\ell}$, and
let $w_1,\ldots,w_m$ be the nodes that lie $n-\ell$ levels above
a node $x$.
We say that $\sigma$ is an \emph{$(\ell,n)$-odd cousins map above $x$}
if the set
\[ \{i\in\{1,\ldots, m \} \, | \, \sigma \text{ acts $\ell$-negatively above } w_i  \} \]
has odd cardinality. Otherwise, we say 
$\sigma$ is an \emph{$(\ell,n)$-even cousins map above $x$}.
\end{enumerate}
\end{defin}
When restricting to the finite subtree $T_n$,
we define odd cousins and even cousins maps $\sigma\in M_{\ell,n}$ similarly.

Note that the two signs $\sgn_{\ell -1}(\sigma,y_0)$
and $\sgn_{\ell -1}(\sigma,y_1)$ in Definition~\ref{def:cousins}(1) must indeed be equal.
To see this, suppose $y_0$ has label $w0$, and $y_1$ has label $w1$.
If $\sgn_1(\sigma,w)=+1$, meaning that
$\sigma(w0)=\sigma(w)0$ and $\sigma(w1)=\sigma(w)1$,
then $\sigma$ permutes the labels (each of length $\ell-1$)
of the $2^{\ell-1}$ nodes above $w0$ with parity given by $\sgn_{\ell-1}(\sigma,y_0)$,
and $\sigma$ separately permutes the labels
of the $2^{\ell-1}$ nodes above $w1$ with parity given by $\sgn_{\ell-1}(\sigma,y_1)$.
Therefore,
\[ \sgn_{\ell-1}(\sigma,y_0) \cdot \sgn_{\ell-1}(\sigma,y_1) = \sgn_{\ell}(\sigma,w)=+1, \]
where the second equality is because $\sigma\in M_{\ell}$.
On the other hand, if $\sgn_1(\sigma,w)=-1$, so that
$\sigma(w0)=\sigma(w)1$ and $\sigma(w1)=\sigma(w)0$,
then define $\lambda$ to be the automorphism of the tree $T_{\ell}$ rooted at $w$
given by $\lambda(w0S)=w1S$ and $\lambda(w1S)=w0S$, for any string of symbols
$S\in\{0,1\}^{\ell-1}$ of length $\ell-1$.
(That is, $\lambda$ simply swaps the two subtrees $T_{\ell-1}$ rooted at $y_0$ and $y_1$.)
Note that $\lambda$ acts $\ell$ levels above $w$
by $2^{\ell-1}$ transpositions of labels of length $\ell$,
and hence $\sgn_{\ell}(\lambda,w)=+1$, since $\ell\geq 2$.
In addition, $\sgn_1(\sigma\lambda)=+1$, so we may apply the
result of the previous case, yielding
\[ \sgn_{\ell-1}(\sigma,y_0) \cdot \sgn_{\ell-1}(\sigma,y_1) =
\sgn_{\ell}(\sigma\lambda,w) = \sgn_{\ell}(\sigma,w)\cdot \sgn_{\ell}(\lambda,w) = +1 . \]
Thus, in either case, we do indeed have
$\sgn_{\ell -1}(\sigma,y_0)=\sgn_{\ell -1}(\sigma,y_1)$.

As for part~(2) of Definition~\ref{def:cousins}, observe 
that in the language of parent and child nodes, the four nodes
$x00$, $x01$, $x10$, and $x11$ lying $2$ levels above a node $x$ form a set of cousins,
since they share a common grandparent $x$.
(More generally, the $2^{\ell}$ nodes that are $\ell$ levels above $x$ form a set
of $(\ell-1)$-th cousins, in this same family tree analogy.)
For any $n\geq \ell$, the $2^n$ nodes that lie $n$ levels above $x$
are naturally partitioned into $m:=2^{n-\ell}$ sets of $2^{\ell}$ nodes, with one
such set above each of the nodes $w_1,\ldots,w_m$
in Definition~\ref{def:cousins}(2); and $\sigma\in M_{\ell}$ must act with even
parity on each of these $m$ sets.

\begin{example}
\label{ex:oddcousins}
We illustrate a $(2,3)$-odd cousins map $\sigma$ above a node $x$
in Figure~\ref{fig:oddcousins}, restricted to $T_3$.
Observe that $\sgn_2(\sigma,0)=\sgn_2(\sigma,1)=+1$
as part of the restriction that $\sigma\in M_2$.
However, $\sigma$ acts $2$-negatively above $0$
(since for both of the children $y$ of $0$,
we have  $\sgn_1(\sigma,y)=-1$), but $2$-positively above $1$.
%whereas for each of the $m=2$ nodes $w$ one level above $1$,
%we have  $\sgn_1(\sigma,w)=$even.
Since $\sigma$ acts $2$-negatively above an odd number of these
$m=2^{3-2}=2$ nodes $w=0,1$, it is indeed
a $(2,3)$-odd cousins map above $x$.
\end{example}

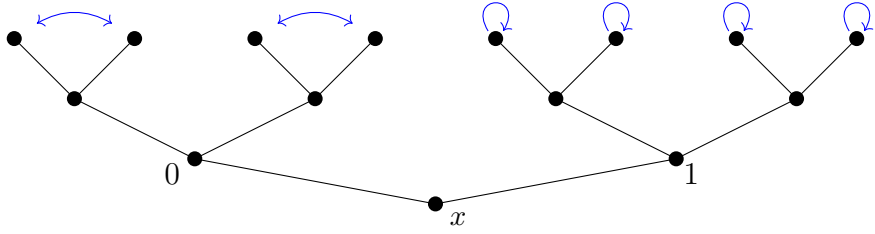
\begin{figure}
\begin{tikzpicture}
\path[draw] (0.8,2.4) -- (1.6,1.6) -- (2.4,2.4);
\path[draw] (4,2.4) -- (4.8,1.6) -- (5.6,2.4);
\path[draw] (7.2,2.4) -- (8,1.6) -- (8.8,2.4);
\path[draw] (10.4,2.4) -- (11.2,1.6) -- (12,2.4);
\path[fill] (0.8,2.4) circle (0.1);
\path[fill] (2.4,2.4) circle (0.1);
\path[fill] (4,2.4) circle (0.1);
\path[fill] (5.6,2.4) circle (0.1);
\path[fill] (7.2,2.4) circle (0.1);
\path[fill] (8.8,2.4) circle (0.1);
\path[fill] (10.4,2.4) circle (0.1);
\path[fill] (12,2.4) circle (0.1);
\path[draw] (1.6,1.6) -- (3.2,0.8) -- (4.8,1.6);
\path[draw] (8,1.6) -- (9.6,0.8) -- (11.2,1.6);
\path[fill] (1.6,1.6) circle (0.1);
\path[fill] (4.8,1.6) circle (0.1);
\path[fill] (8,1.6) circle (0.1);
\path[fill] (11.2,1.6) circle (0.1);
\path[draw] (3.2,0.8) -- (6.4,0.2) -- (9.6,0.8);
\path[fill] (3.2,0.8) circle (0.1);
\path[fill] (9.6,0.8) circle (0.1);
\path[fill] (6.4,0.2) circle (0.1);
%\node (x0) at (6.75,0) {\small $()$};
\node (x0) at (6.7,0) {\small $x$};
\node (a0) at (2.9,0.6) {$0$};
\node (b0) at (9.8,0.6) {$1$};
\draw[<->,blue] (1.1,2.6) to[out=30,in=150] (2.1,2.6);
\draw[<->,blue] (4.3,2.6) to[out=30,in=150] (5.3,2.6);
\draw[->,blue] (7.1,2.5) .. controls (6.8,3) and (7.6,3) .. (7.3,2.5);
\draw[->,blue] (8.7,2.5) .. controls (8.4,3) and (9.2,3) .. (8.9,2.5);
\draw[->,blue] (10.3,2.5) .. controls (10,3) and (10.8,3) .. (10.5,2.5);
\draw[->,blue] (11.9,2.5) .. controls (11.6,3) and (12.4,3) .. (12.1,2.5);
%\node (a00) at (1.25,1.4) {$00$};
%\node (b01) at (5.1,1.4) {$01$};
%\node (a10) at (7.6,1.4) {$10$};
%\node (b11) at (11.5,1.4) {$11$};
%\node (a000) at (0.75,2.7) {$000$};
%\node (a001) at (2.4,2.7) {$001$};
%\node (a010) at (3.95,2.7) {$010$};
%\node (a011) at (5.6,2.7) {$011$};
%\node (a100) at (7.15,2.7) {$100$};
%\node (a101) at (8.75,2.7) {$101$};
%\node (a110) at (10.35,2.7) {$110$};
%\node (a111) at (11.95,2.7) {$111$};
\end{tikzpicture}
\caption{Example~\ref{ex:oddcousins}: a $(2,3)$-odd cousins map above $x$ on $T_3$}
\label{fig:oddcousins}
\end{figure}

The following result is a byproduct of the proof of \cite[Lemma~1.6]{Stoll}, but we include
a self-contained proof here for the convenience of the reader, given our slightly different context.

\begin{prop}
\label{prop:generate}
Let $n\geq 1$, and fix a labeling on the tree $T_n$.
Let $G\subseteq \Aut(T_n)$ be a subgroup.
Suppose the quotient of $G$ formed by restricting to the subtree $T_{n-1}$
is the full group $\Aut(T_{n-1})$. Suppose further that there exists $\sigma_n\in G$
that acts trivially on $T_{n-1}$ and in addition satisfies $\sgn_n(\sigma_n,x_0)=-1$.
Then $G=\Aut(T_n)$.
\end{prop}

%\begin{prop}
%\label{prop:generate}
%Let $n\geq 1$, and fix a labeling on the tree $T_n$.
%Let $G\subseteq \Aut(T_n)$ be a subgroup. Suppose there are
%elements $\sigma_1,\ldots,\sigma_n\in G$ such that for each $1\leq i\leq n$,
%the automorphism $\sigma_i$ acts trivially on the subtree $T_{i-1}$,
%and in addition, $\sgn_i(\sigma_i,x_0)=-1$.
%Then $G=\Aut(T_n)$.
%\end{prop}

\begin{proof}
Let $E$ be the (normal) subgroup of $\sigma\in G$ acting trivially on $T_{n-1}$,
so that $G':=G/E$ is the quotient given by restriction to $T_{n-1}$.
By hypothesis, we have $G'\cong\Aut(T_{n-1})$.
Let $S$ be the set of the $m:=2^{n-1}$ nodes at level $n-1$ of the tree,
and let $\FF_2$ denote the field of 2 elements.
Writing elements of the $m$-dimensional $\FF_2$-vector space $\FF_2^S$
as $v=(v_y)$, where $v_y\in\FF_2$ is the coordinate of $v$ at entry $y\in S$,
the quotient group $G'$ acts on $\FF_2^S$ by $\tau(v)_y = v_{\tau^{-1}(y)}$,
i.e., by permuting the coordinates.

For each $\sigma\in E$, let $w(\sigma)\in \FF_2^S$ be the vector whose $y$-entry
is $0$ if $\sgn_1(\sigma,y)=+1$, or $1$ if $\sgn_1(\sigma,y)=-1$;
that is, $w(\sigma)_y:=(1-\sgn_1(\sigma,y))/2$, viewed as an element of $\FF_2$.
Define $W$ to be the set of vectors $w(\sigma)$, for all $\sigma\in E$.
Then $W$ is a subspace of $\FF_2^S$.
Define
\begin{align}
\label{eq:Vdef}
V :&= \bigg\{ v \in \FF_2^S \, \bigg| \,
\sum_{y\in S} v_y \cdot w(\sigma)_y = 0 \; \forall \, \sigma\in E \bigg\}
%\sum_{y\in S} v_y w(\sigma)_y = 0\text{ for all } \sigma\in E \bigg\}
\\
& =\bigg\{ v \in \FF_2^S \, \bigg| \,
\prod_{y\in S} \sgn_1(\sigma,y)^{v_y} = +1\; \forall \, \sigma\in E \bigg\},
%\prod_{y\in S} \sgn_1(\sigma,y)^{v_y} = 1 \text{ for all } \sigma\in E \bigg\},
\notag
\end{align}
which is the set of all $v\in\FF_2^S$
such that every $\sigma\in E$ is nontrivial above an even number of the nodes
$y\in S$ at which $v_y=1$.

Here, $\cdot$ is the standard dot product, which is a non-degenerate bilinear form on $\FF_2^S$.
Observe that $V$ the orthogonal complement of $W$ in $\FF_2^S$ with respect to $\cdot$,
and hence $V$ is an $\FF_2$-subspace of $\FF_2^S$ satisfying
%, and because it is the orthogonal complement of $W$ with respect to the dot product,
%carved out of $\FF_2^S$ by a set of $\dim W$ linearly independent equations, we have
\[ \dim V + \dim W = |S| .\]
(This identity holds even in characteristic~2, in spite of the fact that $V\cap W$ may be nontrivial.
Indeed, $V$ is the kernel of the $(\dim W)\times |S|$ matrix whose rows are a set of basis vectors for $W$.)
Moreover, $V$ is also
an $\FF_2[G']$-module, because for any $\tau\in G'$ and $v\in V$,
we have$\tau(v)\in V$.
To see this, lift $\tau$ to $G$. Then for any $\sigma\in E$
and $y\in S$, equation~\eqref{eq:Parident} gives
\[ \sgn_1\big(\tau^{-1}\sigma\tau,\tau^{-1}(y)\big)
=\sgn_1(\tau^{-1}, \sigma(y)) \cdot \sgn_1(\sigma,y) \cdot \sgn_1(\tau,\tau^{-1}(y))
= \sgn_1(\sigma,y), \]
where the second equality is because $\sigma(y)=y$, and hence
\[ \sgn_1(\tau^{-1}, \sigma(y)) = \sgn_1(\tau^{-1}, y) =  \sgn_1(\tau,\tau^{-1}(y)) .\]
Therefore,
\[ \prod_{y\in S} \sgn_1(\sigma,y)^{\tau (v)_y}
= \prod_{y\in S} \sgn_1\big(\tau^{-1}\sigma\tau,\tau^{-1}(y)\big)^ {v_{\tau^{-1}(y)}}
= \prod_{y\in S} \sgn_1\big(\tau^{-1}\sigma\tau,y\big)^ {v_y} , \]
which is $+1$ for all $\sigma\in E$, since $v\in V$. Thus, $\tau(v)\in V$;
so $V$ is indeed an $\FF_2[G']$-module.

We claim that $V$ is trivial. To prove the claim, observe
(by the orbit-stabilizer theorem) that for every $v\in V$,
the orbit $G' v$ has cardinality dividing $|G'|=|\Aut(T_{n-1})|$, which is a power of $2$.
Thus, the only way $|G' v|$ can be odd is if $|G' v|=1$, that is, if $v$ is fixed by every
element of $G'$. Observe further that since $G'$ acts transitively on $S$,
the only two elements of $\FF_2^S$ that are fixed by every element of $G'$
are $(0,\ldots,0)$ and $(1,\ldots, 1)$. That is, the only $G'$-orbits in $\FF_2^S$
of odd order are are $\{(0,\ldots 0)\}$ and $\{(1,\ldots 1)\}$.

If $V$ were nontrivial, then $|V|$ would be even, because $V$ is a finite-dimensional
$\FF_2$-vector space.
Partitioning $V$ into $G'$-orbits, and observing that one of those orbits is
the odd-cardinality orbit $\{(0,\ldots 0)\}$, there must be another odd-cardinality orbit in $V$,
and hence $(1,\ldots,1)\in V$.
Therefore, for every $\sigma\in E$, the definition of $V$ yields that
$\sgn_1(\sigma,y) = +1$ for an even number of nodes $y\in S$.
Because each such $\sigma$ fixes all of the nodes at level $n-1$, it follows that
for every $\sigma\in E$, we have $\sgn_n(\sigma,x_0)=+1$,
contradicting the hypothesis about $\sigma_n$.
This contradiction proves our claim.

Thus, we have
\[ \dim W = 0 + \dim W = \dim V + \dim W = |S| = m = 2^{n-1}, \]
and therefore $W$ is the full vector space $\FF_2^S$.
Hence, $E$ is the full group $(\ZZ/2\ZZ)^m$ of all combinations of swaps and non-swaps
above each of the nodes $y\in S$.
Because the quotient $G'=G/E$ is the full group $\Aut(T_{n-1})$,
it follows that $G=\Aut(T_n)$.
\end{proof}

Whereas Proposition~\ref{prop:generate} concerns odd versus even
parities of permutations,
the following concerns only even permutations, but
distinguishing odd cousins from even cousins maps.

\begin{thm}
\label{thm:generate}
Let $n\geq \ell\geq 2$ be integers, fix a labeling on $T_n$,
and let $G\subseteq M_{\ell,n}$ be a subgroup.
Suppose that the quotient of $G$ formed by restricting to the 
subtree $T_{n-1}$ is the full group $M_{\ell,n-1}$.
Suppose further that there exists $\sigma_n\in G$ that acts trivially on $T_{n-1}$
and is an $(\ell,n)$-odd cousins map above $x_0$.
Then $G=M_{\ell,n}$.
\end{thm}

%\begin{thm}
%Fix a labeling on the tree $T_{\infty}$ and an integer $\ell\geq 2$.
%Let $G\subseteq M_\ell$ be a closed subgroup. Suppose there are
%elements $\sigma_1,\sigma_2,\ldots\in G$ such that for each $n\geq 1$,
%the automorphism $\sigma_n$ acts trivially on the subtree $T_{n-1}$,
%and in addition,
%\[ \sgn_n(\sigma_n,x_0)=-1 \quad \text{for } 1\leq n\leq \ell-1, \]
%and 
%\[ \sigma_n \text{ is an $(\ell, n)$-odd cousins map above } x_0
%\quad \text{for } n\geq \ell. \]
%Then $G=M_{\ell}$.
%\end{thm}

\begin{proof}
As in the proof of Proposition~\ref{prop:generate},
let $E$ be the (normal) subgroup of elements $\sigma\in G$ that act trivially on $T_{n-1}$.
Also as before, let $S$ be the set of the $m:=2^{n-1}$ nodes at level $n-1$ of the tree,
let $\FF_2$ denote the field of 2 elements, and for any $v\in\FF_2^S$,
write $v_y\in\FF_2$ for the coordinate of $v$ at entry $y\in S$.
Once again, the quotient $G':=G/E$ given by restriction to $T_{n-1}$
acts on $\FF_2^S$ by permuting the coordinates, and
our hypotheses say that $G'\cong M_{\ell,n-1}$.

Also as in that proof, let
$W$ be the subspace of $\FF_2^S$ consisting of all vectors $w(\sigma)$
given by $w(\sigma)_y:=(1-\sgn_1(\sigma,y))/2$, for all $\sigma\in E$.
Similarly, define the subspace $V$ as in equation~\eqref{eq:Vdef}.
%i.e., as the orthogonal complement of $W$ in $\FF_2^S$.
As before, $V$ and $W$ are $\FF_2[G']$-modules
with $\dim V + \dim W = |S|$.

For each node $u$ at level $n-\ell$ of the tree, let $S_u\subseteq S$
be the set of $2^{\ell-1}$ nodes at level $n-1$ that are above $u$.
Let $X\subseteq \FF_2^S$ be the subspace of vectors that are constant on
each such set $S_u$.
Note that $\dim X=2^{n-\ell}$, since that is the number of nodes at level $n-\ell$.
In addition, $X$ is invariant under the action of $G'$. 

Thus, $G'$ acts on the quotient vector space $\FF_2^S / X$.
As before, for every
\[ [v] :=X+v \in \FF_2^S / X,\]
the orbit $G'[v]$
has cardinality dividing $|G'| = |M_{\ell,n-1}|$, which is a power of $2$.
The $G'$-orbit of the zero vector $[0] \in \FF_2^S / X$ clearly has only one element.

\vspace{6pt}
\noindent\textbf{Claim 1:} There is exactly one nonzero one-element orbit in $\FF_2^S / X$.
%modulo $X$.
\smallskip

To prove Claim~1, we define a particular vector $\tilde{v}\in \FF_2^S$ as follows.
Each node $u$ at level $n-\ell$ of the tree has two children, $u0$ and $u1$,
on level $n-\ell+1$.
Let $S_{u0}$ be the set of $2^{\ell-2}$ nodes at level $n-1$ that lie above $u0$, and 
let $S_{u1}$ be the set of $2^{\ell-2}$ nodes at level $n-1$ that lie above $u1$,
so that $S_u$ is the disjoint union of $S_{u0}$ and $S_{u1}$.
Define $\tilde{v}$ by setting
\[ \tilde{v}_y := \begin{cases}
0 & \text{ if } y\in S_{u0} \text{ for some node } u \text{ at level } n-\ell, \\
1 & \text{ if } y\in S_{u1} \text{ for some node } u \text{ at level } n-\ell,
\end{cases} \]
for each $y\in S$.
% $\tilde{v}_y$ to be $i$ if $y$ belongs to some $S_{ui}$, for $i=0,1$.
%and to be $1$ if $y$ belongs to some $S_{u1}$.
(That is, the coordinates of $\tilde{v}$ consist of alternating blocks of $2^{\ell-2}$ 0's
and of $2^{\ell-2}$ 1's.)
Clearly any $v\in\FF_2^S$ belongs to the coset $[\tilde{v}]$ if and only if $v$ is
constant on each block $S_{ui}$ of $2^{\ell-2}$ nodes,
but not constant on any block $S_u$ of $2^{\ell-1}$ nodes.
In that case, applying any $\sigma\in \Aut(T_{n-1})$ to $v$
--- and hence applying any $\sigma\in G'$ --- yields a vector $\sigma v$ with the same
property. Thus, we have verified that $\sigma [\tilde{v}] = [\tilde{v}]$,
and hence that the orbit $G'[\tilde{v}]$ does indeed have only one element.
Clearly $\tilde{v}\not\in X$, and hence $[\tilde{v}]\neq [0]$.

To finish the proof of Claim~1,
suppose $v'\in\FF_2^S$ has the property that the orbit $G'[v']$ has only one element.
We must show that $[v']$ is either $[0]$ or $[\tilde{v}]$.

We begin by showing
that $v'$ is constant on each block $S_{ui}$ of $2^{\ell-2}$ elements.
To see this, given any node $u$ at level $n-\ell$, pick a node $y\in S_{u0}$.
For any two nodes $z_1,z_2\in S_{u1}$, there is some $\tau\in G'$ that swaps
$z_1$ and $z_2$ but fixes $y$, since $G'$ acts as $\Aut(T_{\ell-1})$
on the copy of $T_{\ell-1}$ rooted at $u$.
Now $\tau v' \in \tau [v'] = [v']$, since the orbit $G'[v']$ has only one element.
The difference $\tau v'- v'$ therefore belongs to $X$ and hence is constant on $S_u$.
But $(\tau v')_y = (v')_y$, so this constant difference is $0$.
Thus, we have
\[ (v')_{z_1} = (\tau v')_{z_1} = (v')_{z_2} .\]
Since this identity holds for all $z_1,z_2\in S_{u1}$, it follows that $v'$ is constant
on $S_{u1}$. By a similar argument, $v'$ is also constant on $S_{u0}$, as desired.

Suppose that there is a node $u$ at level $n-\ell$ for which $v'$ is constant on $S_u$.
Then for any node $t$ at level $n-\ell$, there is some $\tau'\in G$ for which $\tau' u=t$,
because $G'$ acts transitively at that level.
As in the previous paragraph, we have
$\tau'[v']=[v']$, so that $\tau' v' - v'\in X$, and hence $\tau' v'- v'$
is constant on $S_t$. But $\tau' v'$ is also constant on $S_t$, since $v'$ is constant on $S_u$.
Thus, $v'$ itself is constant on $S_t$.
That is, we have shown that if $v'$ is constant on even one block $S_u$ of $2^{\ell-1}$ nodes,
then it is constant on all such blocks.

By the previous two paragraphs, either $v'$ is constant on all blocks $S_u$,
or else on each block $S_u$,
$v'$ takes on one constant value $i\in\{0,1\}$ on $S_{u0}$,
and the other constant value $1-i$ on $S_{u1}$.
In the first of these cases, we have $v'\in X$; in the second, we have $v'-\tilde{v}\in X$.
That is, we have either $[v']=[0]$ or $[v']=[\tilde{v}]$, proving Claim~1.

%Having proven our first claim, we make a second claim:

\vspace{6pt}
\noindent\textbf{Claim 2:} The space $V$ is $V = X$, i.e. $V/X$ is trivial.
\smallskip

To see this, suppose not. Then $|V/X|$ is even,
since it is a nontrivial power of $2$.
Yet again following the proof of Proposition~\ref{prop:generate},
partition $V/X$ into $G$-orbits. Since one of those orbits is the single-element
$G'[0]$, some other orbit in $V$ must also have odd cardinality.
But as we noted, all orbits are a power of $2$ in cardinality, and hence $V/X$
must contain the other single-element orbit $[\tilde{v}]$. That is,
$\tilde{v}\in V$. Our hypotheses state that $E$ contains
an $(\ell,n)$-odd cousins map $\lambda$. In particular,
there are an odd number of sets $S_u$ for which $\lambda$
is nontrivial above an odd number of nodes of $S_{u1}$.
It follows that
\[ \prod_{y\in S} \sgn_1(\lambda,y)^{\tilde{v}_y}
= \prod_{u} \prod_{y\in S_{u1}} \sgn_1(\lambda,y) = -1, \]
where the second product is over all nodes $u$ at level $n-\ell$.
By definition of $V$, it follows that $\tilde{v}\not\in V$,
a contradiction.

Claim~2 follows: we have $V=X$.
Hence,
\[ \log_2|E| = \dim W = |S| - \dim V = |S| - \dim X = 2^{n-1} - 2^{n-\ell}, \]
and therefore
\[ \log_2 |G| = \log_2 |E| + \log_2 |M_{\ell,n-1}|
= \big( 2^{n-1} - 2^{n-\ell} \big) +  \big( 2^{n-1} - 2^{n-\ell} \big)
= 2^n - 2^{n-\ell + 1} = \log_2 |M_{\ell,n}| \]
by equation~\eqref{eq:Msize}.
Since $G\subseteq M_{\ell,n}$, it follows that
$G=M_{\ell,n}$.
\end{proof}

\begin{thm}
\label{thm:abelian}
Fix a labeling on the tree $T_{\infty}$ and an integer $\ell\geq 2$.
Let $n\geq 0$ be an integer, and
define $\psi_n:M_{\ell,n}\to \{\pm 1 \}^n$ by
$\psi_n(\sigma):= (e_1, \ldots, e_n)$, where for each $i=1,\ldots,n$, we have
\[ e_i:= \begin{cases}
\sgn_i(\sigma,x_0) & \text{ if } 1\leq i\leq \ell -1, \\
+1 & \text{ if } i\geq \ell \text{ and }
\sigma \text{ is an $(\ell,i)$-even cousins map above } x_0, \\
-1 & \text{ if } i\geq \ell \text{ and }
\sigma \text{ is an $(\ell,i)$-odd cousins map above } x_0.
\end{cases} \]
Then $\psi_n$ is a surjective group homomorphism,
and $\ker\psi_n$ is the commutator subgroup of $M_{\ell,n}$.
In particular, the abelianization $M_{\ell,n}^{\textup{ab}}$
is isomorphic to $\{\pm 1 \}^n$.
\end{thm}

\begin{proof}
It is straightforward to see that
the composition of two $(\ell,i$)-even cousins maps,
or of two $(\ell,i$)-odd cousins maps, is an $(\ell,i)$-even cousins map.
Similarly, the composition of one $(\ell,i)$-even and one $(\ell,i$)-odd cousins map
is an $(\ell,i$)-odd cousins map.
It is also well known that the composition of two like-sign permutations is even,
and of two opposite-sign permutations is odd. Thus, $\psi_n$ is a group homomorphism.
In addition, any permutation and its inverse have the same sign,
and similarly for even-cousins and odd-cousins maps.
%and the inverse of an even-cousins map is even-cousins,
%and similarly for odd-cousins maps.
Therefore, for any $\sigma,\tau\in M_{\ell,n}$,
the commutator $[\sigma,\tau]:=\sigma\tau\sigma^{-1}\tau^{-1}$ must lie in $\ker\psi_n$.

For the rest of the proof, we proceed by induction on $n\geq 0$.
The desired result is trivially true for $n=0$.
For each $n\geq 1$, assume the statement is true for $n-1$, and
let $E_n\subseteq M_{\ell,n}$ be the normal subgroup of elements
of $M_{\ell,n}$ acting trivially on the first $n-1$ levels of the tree.
For $n\leq\ell-1$, we have $M_{\ell,n}=\Aut(T_n)$, and hence
half of the elements of $E_n$ are even, and half are odd.
Similarly, for $n\geq \ell$, half of the elements of $E_n$ are $(\ell,n)$-even cousins maps,
and the other half are $(\ell,n)$-odd cousins maps.
We also have
$M_{\ell,n} / E_n \cong M_{\ell,n-1}$ by restricting to the subtree $T_{n-1}$.

To see that $\psi_n$ is surjective, consider an arbitrary $(e_1,\ldots, e_n)\in \{\pm 1\}^n$.
By our inductive assumption of surjectivity, there is
some $\overline{\sigma}\in M_{\ell,n-1}$ such that
\[ \psi_{n-1}(\overline{\sigma})=(e_1,\ldots, e_{n-1}) .\]
Lift $\overline{\sigma}$ to some $\sigma\in M_{\ell,n}$,
so that
\[ \psi_n(\sigma)=(e_1,\ldots,e_{n-1}, \tilde{e}_n) \quad \text{for some} \quad \tilde{e}_n\in\{\pm 1\} .\]
If $\tilde{e}_n=e_n$, we are done. Otherwise, if $n\leq \ell -1$, then
pick $\tau\in E_n$ that is odd; or if $n\geq \ell$,
pick $\tau\in E_n$ that is an odd cousins map.
Then $\psi_n(\tau \sigma) = (e_1,\ldots, e_n)$, as desired.

It remains to show,
given $\sigma\in\ker(\psi_n)$, that $\sigma$ belongs to
the commutator subgroup of $M_{\ell,n}$. Restricting $\sigma$ to $T_{n-1}$
yields $\overline{\sigma} \in M_{\ell,n-1}$, which by our inductive hypothesis is a product
of commutators
%\[ \sigma' = \prod_{i=1}^N \bar{\rho}_i \bar{\tau}_i \bar{\rho}_i^{-1} \bar{\tau}_i^{-1} \]
\[ \overline{\sigma} = \prod_{i=1}^N \big[ \bar{\rho}_i , \bar{\tau}_i \big] \]
with $\bar{\rho}_i , \bar{\tau}_i \in M_{\ell,n-1}$. Lifting $\bar{\rho}_i$
and $\bar{\tau}_i$ to $\rho_i,\tau_i\in M_{\ell,n}$, we have that
%\[ \sigma \bigg[ \prod_{i=1}^N \rho_{i} \tau_{i} \rho_{i}^{-1} \tau_{i}^{-1} \bigg]^{-1} \in \ker(\psi_n) \cap E_n . \]
\[ \bigg( \prod_{i=1}^N \big[ \rho_{i} , \tau_{i} \big] \bigg)^{-1} \sigma \in \ker(\psi_n) \cap E_n . \]
If the expression above is a product of commutators, then so is $\sigma$.
Thus, we may assume without loss that $\sigma\in \ker(\psi_n) \cap E_n$.

Recall that the two children of $x_0$ are the nodes labeled $0$ and $1$.
Write
\[ \sigma=(\sigma_0,\sigma_1)\in E_{n-1}\times E_{n-1} , \]
where $\sigma_i$ describes the action of $\sigma$ on the copy of $T_{n-1}$
rooted at node $i$. In what follows, we will say that $\sigma_i$ is even 
if it is an even permutation of the $2^{n-1}$ nodes at the top level of the tree $T_{n-1}$
that it acts on; otherwise, we will say $\sigma_i$ is odd.
We consider several cases.

First, suppose that
$n\leq \ell-1$ and both $\sigma_0$ and $\sigma_1$ are even.
Then by our inductive hypothesis, both $\sigma_0$ and $\sigma_1$
are products of commutators in $M_{\ell,n-1}$.
For each commutator $[\rho, \tau]$ in the product for $\sigma_0$,
we may define
\[ \tilde{\rho} :=(\rho,e) \quad\text{and}\quad \tilde{\tau} :=(\tau,e) , \]
both of which are elements of $\Aut(T_n)=M_{\ell,n}$.
Thus, $(\sigma_0,e)$ is a product of commutators $[\tilde{\rho}, \tilde{\tau}]$.
By similar reasoning, $(e,\sigma_1)$ is also a product of commutators in $M_{\ell,n}$,
and hence so is $\sigma = (\sigma_0,e)(e,\sigma_1)$.

Second, suppose that $n=\ell$,
in which case $\sigma_0$ and $\sigma_1$ are both necessarily even,
since we assumed that $e_\ell=+1$, and hence that
$\sigma$ is an even-cousins map.
Again by our inductive hypothesis, both $\sigma_0$ and $\sigma_1$
are products of commutators in $M_{\ell,\ell-1}$.
Fix $\theta\in E_{\ell-1}$ that is odd.
For each commutator $[\rho, \tau]$ in the product for $\sigma_0$, define
\[ \tilde{\rho}:=\begin{cases}
(\rho,e) & \text{ if }\rho \text{ is even}, \\
(\rho,\theta) & \text{ if }\rho \text{ is odd}.
\end{cases} \]
Then $\tilde{\rho}\in M_{\ell,\ell}$,
because our definition ensures that $\sgn_{\ell}(\tilde{\rho},x_0)=+1$.
Define $\tilde{\tau}\in M_{\ell,\ell}$ similarly.
Then $(\sigma_0,e)$ is a product of the commutators $[\tilde{\rho} , \tilde{\tau} ]$,
because any appearances of $\theta$ in the second coordinate will cancel within the individual commutators.
Similarly, $(e,\sigma_1)$ is also a product of commutators in $M_{\ell,\ell}$,
and hence so is $\sigma = (\sigma_0,e)(e,\sigma_1)$.

Third, suppose that $n\geq \ell+1$
and that both $\sigma_0$ and $\sigma_1$ are even-cousins maps in $M_{\ell,n-1}$.
Then both $\sigma_0$ and $\sigma_1$ are
products of commutators in $M_{\ell,n-1}$, by our inductive hypothesis.
Recalling the automorphism $\theta\in E_{\ell-1}$ from the previous case,
define $\theta_{n-1}\in M_{\ell,n-1}$ to be the automorphism of $T_{n-1}$ given by
\[ \theta_{n-1}(yw):=\theta(y) w \]
for each node $y$ at level $\ell-1$ and each word $w\in \{0,1\}^{n-\ell}$.
(That is, $\theta_{n-1}$ acts like $\theta$ at level $\ell-1$ but makes no
further permutations of the labels above that level.)
For each commutator $[\rho, \tau]$ in the product for $\sigma_0$, define
\[ \tilde{\rho}:=\begin{cases}
(\rho,e) & \text{ if } \sgn_{\ell-1}(\rho,0)=+1, \\
(\rho,\theta_{n-1}) & \text{ if }\sgn_{\ell-1}(\rho,0)=-1.
\end{cases} \]
As in the previous case, our definition ensures that
$\sgn_{\ell}(\tilde{\rho},x_0)=+1$, so that $\tilde{\rho}\in M_{\ell,n}$.
Define $\tilde{\tau}\in M_{\ell,n}$ similarly.
As in the previous two cases, using the commutators
$[\tilde{\rho} , \tilde{\tau} ]$,
it follows that  $\sigma = (\sigma_0,e)(e,\sigma_1)$ is a product of commutators in $M_{\ell,n}$.

Finally, suppose either that 
$n\leq \ell-1$, and both $\sigma_0$ and $\sigma_1$ are odd;
or that $n\geq \ell+1$,
and both $\sigma_0$ and $\sigma_1$ are odd-cousins maps.
Let $\lambda\in M_{\ell,n}$ be the automorphism given by
$\lambda(0w)=1w$ and $\lambda(1w)=0w$, for every word $w\in \{0,1\}^{n-1}$.
(That is, $\lambda$ simply exchanges the two halves of the tree.)
It is immediate from the definition of the group that $\lambda\in M_{\ell,n}$,
since $\sgn_{\ell}(\lambda,y)=+1$ for every node $y$ of the tree.

If $n\leq \ell-1$, choose $\rho\in E_{n-1}$ that is odd;
of if $n\geq \ell+1$, choose $\rho\in E_{n-1}$ that is an odd-cousins map.
Define $\mu$ to be the commutator
\[ \mu:=[\lambda, (\rho,e)]=(\rho^{-1},\rho)\in E_n, \]
which is odd (respectively, an odd-cousins map) on each half of the
$n$-th level of the tree, if $n\leq\ell-1$ (respectively, if $n\geq \ell+1$).
Thus,
\[ \sigma\mu^{-1} = (\sigma_0\rho, \sigma_1 \rho^{-1}) \in E_n \]
is even (respectively, an even-cousins map) on each half of the tree.
By the first  (respectively, third) case above,
$\sigma\mu^{-1}$ is a product of commutators.
Therefore, $\sigma$ is also such a product, because $\mu$ is a commutator.
\end{proof}

\section{Attaining the full group}
\label{sec:surj}

Before proving Theorem~\ref{thm:main2},
we must define the quantities $\kappa_1,\kappa_2,\ldots$
referred to in the statement of that result.
To do so, we recall that the \emph{cross ratio} of four points $a,b,c,d\in\PKbar$ is
\[ \CR(a,b,c,d) := \frac{(a-b)(c-d)}{(a-c)(b-d)}, \]
with the usual understanding of what this expression means if any one of the four points
is $\infty$, i.e. the two terms containing $\infty$ cancel.
(For example, if $c=\infty$, the cross ratio above is $-(a-b)/(b-d)$.)

We also set the following notation throughout this section.
As in Theorems~\ref{thm:main1} and~\ref{thm:main2},
assume that the two critical points $\xi_1,\xi_2\in\PKbar$
of the quadratic rational function $f\in K(z)$
collide at the $\ell$-th iterate, for some $\ell\geq 2$.
Observe that $\xi_1$ and $\xi_2$ cannot both be periodic;
otherwise, if $n\geq\ell$ is a multiple of both periods, we would have
\[ \xi_1 =f^n(\xi_1)=f^n(\xi_2)=\xi_2, \]
a contradiction.
Thus, it cannot be that each of $\xi_1$ and $\xi_2$ is in the forward orbit
of the other. Without loss, then, we may assume that
$\xi_2$ is not in the forward orbit of $\xi_1$.

%Exchanging the subscripts if necessary, assume that with $\xi_1\neq \infty$.
%Assume further that $\xi_1$ and $\xi_2$ are not preperiodic.
Set the following notation:
\begin{tabbing}
\hspace{8mm} \= \hspace{21mm} \=  \kill
\> $\delta\in K^{\times}$: \> the discriminant of the minimal polynomial of $\xi_1$ over $K$, \\
\> \> or $\delta=1$ if $\xi_1\in\PK$ \\
\> $L$: \> the field $L:= K(\sqrt{\delta})$ \\
\> $F=(P,Q)$: \> a homogeneous lift of $f$, with $P,Q\in K[X,Y]$ \\
\> $(s_0,t_0)$: \> a lift of the point $x_0\in\PK$ to $K\times K \smallsetminus \{(0,0)\}$ \\
\> $(\eta_i,\theta_i)$: \> for each $i=1,2$, a lift of the critical point $\xi_i$ to
$L \times L \smallsetminus \{(0,0)\}$ \\
\> $(P_n,Q_n)$: \> for each integer $n\geq 1$, write $(P_n,Q_n)=F^n$ \\
\> $H_n$: \> the polynomial $H_n:=t_0 P_n - s_0 Q_n$, as in Corollary~\ref{cor:iterdisc}
\end{tabbing}
In particular, $P,Q\in K[X,Y]$ are relatively prime homogeneous polynomials of degree $2$
such that $f(z) = P(z,1)/Q(z,1)$.
In addition, we have $x_0=s_0/t_0$, with the usual understanding of $1/0$ as the point $\infty$.

\begin{defin}
\label{def:kappa}
With notation as above, define $\kappa_n\in L=K(\sqrt{\delta})$ by
\[ \kappa_n:= \begin{cases}
\Delta(H_n) & \text{ if } 1\leq n \leq \ell -1, \\[1mm]
\dsps \CR\big( x_0, f^{n-\ell + 1}(\xi_1), f^n(\xi_2), f(\xi_2)\big)
& \text{ if } n\geq \ell+1,
\end{cases} \]
and
\[ \kappa_\ell := \begin{cases}
\Delta(H_\ell) & \text{ if } \sqrt{\delta}\not\in K, \\[1mm]
%\CR\big( f(\infty), f^{\ell-1}(\infty), f^{\ell-1}(0), \infty \big)
\dfrac{f(\xi_2) - f^{\ell-1}(\xi_2)}{f(\xi_2) - f^{\ell-1}(\xi_1)}
\cdot \CR\big( x_0, f(\xi_1), f^{\ell}(\xi_2), f(\xi_2) \big)
& \text{ if } \sqrt{\delta}\in K \text{ and }\ell\geq 3, \\[1mm]
\Delta(\theta_2 P - \eta_2 Q) \cdot \CR\big( x_0, f(\xi_1), f^{2}(\xi_2), f(\xi_2) \big)
& \text{ if } \sqrt{\delta}\in K \text{ and }\ell=2.
\end{cases}
\]
If the above formulas would result in $\kappa_n=\infty$,
then re-define $\kappa_n:=0$.
%\[ \kappa_n:= \begin{cases}
%\Delta(H_n) & \text{ if } 1\leq n \leq \ell -1, \\[1mm]
%\dsps \CR\big( x_0, f^{n-\ell + 1}(\xi_1), f^n(\xi_2), f(\xi_2)\big) & \text{ if } n\geq \ell,
%\end{cases} \quad \text{ if } \sqrt{\delta}\in K, \]
%or
%\[ \kappa_n:= \begin{cases}
%\Delta(H_n) & \text{ if } 1\leq n \leq \ell, \\[1mm]
%\dsps \CR\big( x_0, f^{n-\ell + 1}(\xi_1), f^n(\xi_2), f(\xi_2)\big) & \text{ if } n\geq \ell+1,
%\end{cases} \quad \text{ if } \sqrt{\delta}\not\in K, \]
\end{defin}

The final case of Definition~\ref{def:kappa},
that the original formula would give $\kappa_n=\infty$,
can only arise from one of the cross ratio terms, and then only if
$f^n(\xi_2)=x_0$.

The special case $n=\ell$ in Definition~\ref{def:kappa}
corresponds, not coincidentally,
to the lowest level of the tree at which the groups
$M_{\ell}$ and $\widetilde{M}_{\ell}$ differ.
Indeed, as we saw in the proof of Theorem~\ref{thm:main1},
the discriminant $\Delta(H_\ell)$ is a square in $K$ if and only if $\delta$
is also a square in $K$.

Note also that if we replace the lifts $F$ and $(s_0,t_0)$
of $f$ and $x_0$ by other lifts
$\tilde{F}$ and $(\tilde{s}_0,\tilde{t}_0)$,
the effect is to multiply each $H_n$ by some $c \in K^{\times}$.
The discriminant $\Delta(H_n)$ is then multiplied by an even power of $c$,
and hence by a square in $K^{\times}$.
Similar reasoning applies to $\Delta(\theta_2 P - \eta_2 Q)$
in the case that $\sqrt{\delta}\in K$.
Thus, although a different choice of lift may change the exact values
of $\kappa_1,\ldots,\kappa_{\ell}$, it does not change whether
any of the products $\kappa_{i_1}\cdots \kappa_{i_m}$ of Theorem~\ref{thm:main2}
are squares.
This condition is also not affected by $K$-rational coordinate changes,
as the next result shows.

\begin{prop}
\label{prop:coordchange}
%Let $f\in K(z)$ be a rational function of degree~2
%with critical points $\xi_1,\xi_2\in\PKbar$
%that collide at the $\ell$-th iterate, for some $\ell\geq 2$.
%Fix a point $x_0\in\PK$, and
%let $\kappa_1,\kappa_2,\ldots\in$ be the associated quantities. Let
With notation as above, let
$\nu\in\PGL(2,K)$, let $g:=\nu\circ f \circ \nu^{-1}$, and let $y_0:=\nu(x_0)$.
Let $\tilde{\kappa}_1,\tilde{\kappa}_2,\ldots\in L= K(\sqrt{\delta})$ be the associated quantities
for the preimages of $y_0$ under iterates of $g$.
Then for each $n\geq 1$, there exists $c_n\in K^{\times}$
such that $\tilde{\kappa}_n = c_n^2 \kappa_n$.
\end{prop}

\begin{proof}
Observe that the $K$-rational coordinate change $\nu$ does not change the discriminant
$\delta$ of the critical points, except possibly by a factor of a square in $K$.
In particular, the field $K(\sqrt{\delta})$ is the same for $g$ as for $f$,
and the choice of which of the three formulas used to define $\kappa_{\ell}$
also does not change.

In addition, cross ratios are well known to be unaffected by coordinate change,
as is easy to check by hand.
Hence, it suffices to show that the discriminants $\Delta(H_n)$
and (when $\sqrt{\delta}\in K$) $\Delta(\theta_2 P - \eta_2 Q)$ are affected
only by square factors in $K$ under coordinate changes.
We will prove this fact for $\Delta(H_n)$; the proof for 
$\Delta(\theta_2 P - \eta_2 Q)$ is the same with $n=1$.

Lift $f$ to $F=(P,Q)$ and $x_0$ to $(s_0,t_0)$
as in the notation presented just before Definition~\ref{def:kappa}.
Lift $\nu$ to $N=(R,S)$, where
\[ R=aX+bY \;\;\text{and}\;\; S=cX+dY,
\quad\text{with}\quad a,b,c,d\in K \;\;\text{and}\;\; \eps:=ad-bc\in K^{\times} .\]
Then $G:=N\circ F \circ N^{-1}$ is a lift of $g$, and
$(u_0,v_0):=N(s_0,t_0)$ is a lift of $y_0$.

For any $n\geq 1$, write $F^n=(P_n,Q_n)$ and $G^n=(\tilde{P}_n,\tilde{Q}_n)$,
so that
\[ \tilde{P}_n = aP_n \circ N^{-1} + bQ_n \circ N^{-1}
\quad\text{and}\quad
\tilde{Q}_n =cP_n \circ N^{-1} + dQ_n \circ N^{-1} ,\]
Therefore, writing $H_n:= t_0 P_n - s_0 Q_n$ and
$\tilde{H}_n:= v_0 \tilde{P}_n - u_0 \tilde{Q}_n$, we have
\begin{align*}
\tilde{H}_n &= (c s_0 + d t_0)(a P_n\circ N^{-1} + b Q_n\circ N^{-1})
- (a s_0 + b t_0)(c P_n\circ N^{-1} + d Q_n\circ N^{-1}) \\
&= ad t_0 P_n\circ N^{-1} + bc s_0 Q_n\circ N^{-1}
- bc t_0 P_n\circ N^{-1} - ad s_0 Q_n\circ N^{-1} \\
& = (ad-bc) H_n\circ N^{-1} = \eps H_n\circ N^{-1}.
\end{align*}
Write $N^{-1}=(R',S')=(a'X + b'Y, c'X+d'Y)$.
Then a simple computation shows
\[ \Res(R',S') = a'd'-b' c' = (ad-bc)^{-1} = \Res(R,S)^{-1} = \eps^{-1}.\]
Hence, applying Theorem~\ref{thm:compdisc} with $J=H_n$
to the composition $H_n \circ N^{-1}$, and writing $m:=\deg(H_n)=2^n$,
we have
\[ \Delta(H_n \circ N^{-1})
= (-1)^0 \Delta(H_n)^1 (\eps^{-1})^{m(m-2)} \Res(H_n, \eps^{-1})
= \eps^{-m(m-1)} \Delta(H_n) ,\]
where the second equality is because $\eps^{-1}$ is a constant, and hence
\[ \Res(H_n, \eps^{-1}) = ( \eps^{-1} )^{\deg(H_n)} = \eps^{-m}. \]
Combining the above computations, we have
\[ \Delta(\tilde{H}_n) = \Delta(\eps H_n \circ N^{-1})
= \eps^{2m-2} \Delta(H_n \circ N^{-1})
= \eps^{-(m-2)(m-1)}\Delta(H_n), \]
and therefore the result follows by choosing $c_n:= \eps^{-(m-2)(m-1)/2}\in K^{\times}$.
\end{proof}

Before proving Theorem~\ref{thm:main2}, we also need
two important technical results.

\begin{lemma}
\label{lem:Qn}
Let $\ell\geq 2$, and suppose that
\begin{equation}
\label{eq:ABCform}
f(z) = \frac{Az^2 + B}{z^2 +C}
\quad \text{for some }
A,B,C\in K \text{ with } AC-B\neq 0 .
\end{equation}
Suppose further that the two critical points $0,\infty$ of $f$
collide at the $\ell$-th iterate,
and that $\infty$ is not in the forward orbit of $0$.
%and that they are not preperiodic.
%Let $x\in\PKbar$ not in the forward orbit of $f(0)$ or $f(\infty)$,
Let $x\in\PKbar$ not be in the forward orbit of $f(\infty)$,
and write $f^{-1}(x)=\{\pm y\}$ for some $y\in \Kbar$.
Define $m:= 2^{\ell -2}$, and write 
\[ f^{-(\ell-1)}(y) = \{ \pm \alpha_1, \ldots, \pm \alpha_m \}
\quad\text{and}\quad
f^{-(\ell-1)}(-y) = \{ \pm \alpha'_1, \ldots, \pm \alpha'_m \} , \]
for some $\alpha_i, \alpha'_i\in\Kbar$.
Then, possibly after reversing the roles of $\alpha_m$ and $-\alpha_m$, we have
\[ \bigg( \prod_{i=1}^m \alpha_i + \prod_{i=1}^m \alpha'_i \bigg)^2
= 4 q_{\ell-1} \cdot \CR\big( x, f(0), f^{\ell}(\infty), f(\infty) \big), \]
where
\[ q_{\ell-1} := (-C)^{2^{\ell-2}} \prod_{i=2}^{\ell -1}
\bigg( \frac{f(\infty) - f^i(\infty)}{ f(\infty) - f^i(0)} \bigg)^{2^{\ell-i-1}} \in K^{\times} .\]
\end{lemma}

\begin{proof}
Note that for any points $s,t\in\PKbar$ with $f(s)=t$, we do indeed have
$f^{-1}(t)=\{\pm s\}$. In particular, the notation
$f^{-1}(x)=\{\pm y\}$ in the statement of the lemma makes sense,
as does the notation for $f^{-(\ell-1)}(\pm y)$;
see Figure~\ref{fig:QnLemma}.

\begin{figure}
\begin{tikzpicture}
\path[draw, dashed] (0.2,2.6) -- (3.2,0.8) -- (6.2,2.6);
\path[draw, dashed] (8.2,2.6) -- (11.2,0.8) -- (14.2,2.6);
\path[draw] (3.2,0.8) -- (7.2,0.2) -- (11.2,0.8);
\path[fill] (3.2,0.8) circle (0.1);
\path[fill] (11.2,0.8) circle (0.1);
\path[fill] (7.2,0.2) circle (0.1);
\node (x0) at (6.9,0) {\small $x$};
\node (a0) at (2.9,0.6) {$y$};
\node (b0) at (11.6,0.6) {$-y$};
\path[fill] (0.2,2.6) circle (0.1);
\path[fill] (0.8,2.6) circle (0.1);
\path[fill] (1.8,2.6) circle (0.1);
\path[fill] (2.4,2.6) circle (0.1);
\path[fill] (5.6,2.6) circle (0.1);
\path[fill] (6.2,2.6) circle (0.1);
\node (mal1) at (0.1,2.9) {\footnotesize $\alpha_1$};
\node (al1) at (0.8,2.9) {\footnotesize $-\alpha_1$};
\node (mal2) at (1.7,2.9) {\footnotesize $\alpha_2$};
\node (al2) at (2.45,2.9) {\footnotesize $-\alpha_2$};
\node (dots1) at (4.1,2.6) {\footnotesize $\cdots$};
\node (malm) at (5.45,2.9) {\footnotesize $\alpha_m$};
\node (alm) at (6.25,2.9) {\footnotesize $-\alpha_m$};
\path[fill] (8.2,2.6) circle (0.1);
\path[fill] (8.8,2.6) circle (0.1);
\path[fill] (9.8,2.6) circle (0.1);
\path[fill] (10.4,2.6) circle (0.1);
\path[fill] (13.6,2.6) circle (0.1);
\path[fill] (14.2,2.6) circle (0.1);
\node (pmal1) at (8.1,2.92) {\footnotesize $\alpha'_1$};
\node (pal1) at (8.8,2.92) {\footnotesize $-\alpha'_1$};
\node (pmal2) at (9.7,2.92) {\footnotesize $\alpha'_2$};
\node (pal2) at (10.4,2.92) {\footnotesize $-\alpha'_2$};
\node (dots2) at (12.1,2.6) {\footnotesize $\cdots$};
\node (pmalm) at (13.45,2.92) {\footnotesize $\alpha'_m$};
\node (palm) at (14.25,2.92) {\footnotesize $-\alpha'_m$};
\end{tikzpicture}
\caption{Lemma~\ref{lem:Qn}}
\label{fig:QnLemma}
\end{figure}

Observe that $\infty$ cannot be periodic; if it were, then 
choosing $n\geq \ell$ to be a multiple of the period,
we would have $f^n(0)=f^n(\infty)=\infty$, contradicting the hypotheses.
In addition, we have $f^n(0)\neq f(\infty)$ for all $n\geq 0$, because
$\infty$ itself is the only immediate preimage of $f(\infty)$.
Similarly, we also have $f^n(\infty)\neq f(\infty)$ for all $n\geq 2$.

Generalize the formula for $q_{\ell-1}$ by defining
\[ q_n := (-C)^{2^{n-1}} \prod_{i=2}^{n}
\bigg( \frac{f(\infty) - f^i(\infty)}{ f(\infty) - f^i(0)} \bigg)^{2^{n-i}} \]
for each $n\geq 1$. By the previous paragraph, the numerators and
denominators in this product all lie in $K^{\times}$.
Similarly, we have $C\in K^{\times}$; for if $C=0$, then we would have $f(0)=\infty$.
Thus, we have $q_n\in K^{\times}$.
%Equivalently, that is, we have
%$q_1:=-C$, and for each $n=2,\ldots, \ell-1$, 
%\[ q_n := q_{n-1}^2 \bigg( \frac{f^n(\infty) - f(\infty)}{f^n(0) - f(\infty)}\bigg) .\]

Observe that for any $a,b,c\in\Kbar$, we have
\begin{equation}
\label{eq:CRrewrite}
\frac{a^2-b^2}{a^2-c^2} = \CR(a^2,b^2,c^2,\infty)
= \CR\big( f(a), f(b), f(c), f(\infty) \big),
\end{equation}
since $f(z)$ is a linear fractional transformation applied to $z^2$,
and because, as noted in the proof of Proposition~\ref{prop:coordchange}, cross ratios
are unchanged under linear fractional transformations.

We claim that for each $n\geq 1$ and for any point $w\in\PKbar$
not in the forward orbit of $f(\infty)$, we have
\begin{equation}
\label{eq:lemmaclaim}
\prod_{i=1}^{2^{n-1}} \beta_i^2 = q_n \cdot \bigg( \frac{w-f^n(0)}{w- f^n(\infty)}\bigg),
\quad\text{where} \quad
f^{-n}(w) = \{\pm \beta_1,\ldots, \pm \beta_{2^{n-1}} \} .
\end{equation}
(The assumption that $w$ is not in the forward orbit of $f(\infty)$
guarantees that both sides of equation~\eqref{eq:lemmaclaim}
are defined and finite.)
To prove this identity for $n=1$, write $f^{-1}(w)=\{\pm\beta_1\}$.
Solving the equation $f(\beta_1)=w$ for $\beta_1^2$ gives
\[ \beta_1^2 = \frac{Cw-B}{A-w} = -C \bigg( \frac{ w - B/C}{w- A} \bigg)
= -C \bigg( \frac{ w - f(0)}{w- f(\infty)} \bigg), \]
verifying the claim for $n=1$, since $q_1=-C$.
Proceeding inductively, for any $n\geq 2$, suppose the claim is true for $n-1$.
Write $f^{-1}(w)=\{\pm u\}$ for some $u\in\Kbar$, with
\[ f^{-(n-1)}(u)=\{\pm \beta_1,\ldots, \pm \beta_{2^{n-2}}\}
\quad\text{and}\quad
f^{-(n-1)}(-u)=\{\pm \beta_{1+2^{n-2}},\ldots, \pm \beta_{2^{n-1}}\} .\]
Then by our inductive hypothesis, we have
\begin{align*}
\prod_{i=1}^{2^{n-1}} \beta_i^2 &=
\bigg[ \prod_{i=1}^{2^{n-2}} \beta_i^2 \bigg]
\cdot \bigg[\prod_{i=1+2^{n-2}}^{2^{n-1}} \beta_i^2 \bigg] =
\bigg[ q_{n-1} \cdot \bigg( \frac{u-f^{n-1}(0)}{u - f^{n-1}(\infty)} \bigg) \bigg]
\bigg[ q_{n-1} \cdot \bigg( \frac{-u-f^{n-1}(0)}{-u - f^{n-1}(\infty)} \bigg) \bigg]  \\
& = q_{n-1}^2 \cdot \bigg( \frac{u^2 - (f^{n-1}(0))^2}{u^2 - (f^{n-1}(\infty))^2} \bigg)
= q_{n-1}^2 \cdot \CR\big(w,f^n(0),f^n(\infty),f(\infty) \big)
\\
&= q_{n-1}^2 \cdot \bigg(\frac{f(\infty) - f^n(\infty)}{f(\infty) - f^n(0)} \bigg)
\bigg( \frac{w-f^n(0)}{w-f^n(\infty)} \bigg)
= q_n \cdot \bigg( \frac{w-f^n(0)}{w-f^n(\infty)} \bigg),
\end{align*}
where the fourth equality is by equation~\eqref{eq:CRrewrite}
and the fact that $f(u)=w$,
and the sixth is by definition of $q_n$.
Thus, we have proven the claim of equation~\eqref{eq:lemmaclaim}.

Returning to the notation of the statement of the lemma, the claim gives
\[ \bigg( \prod_{i=1}^m \alpha_i \cdot \prod_{i=1}^m \alpha'_i \bigg)^2
= q_\ell \bigg( \frac{x - f^{\ell}(0)}{x - f^{\ell}(\infty)} \bigg)
= q_{\ell} = q_{\ell-1}^2 \bigg( \frac{ f(\infty) - f^{\ell}(\infty)}{ f(\infty) - f^{\ell}(0)} \bigg)
= q_{\ell-1}^2, \]
where the second and fourth equalities are because $f^{\ell}(0)=f^{\ell}(\infty)$,
and the third is by definition of $q_n$. Thus, we have 
\[ \prod_{i=1}^m \alpha_i \cdot \prod_{i=1}^m \alpha'_i
= \pm q_{\ell-1}, \]
and by switching the roles of $\alpha_m$ and $-\alpha_m$ if necessary, we may
assume that this product is in fact $q_{\ell-1}$.
In addition, because $f^{\ell}(0)=f^{\ell}(\infty)$ but $f^{\ell-1}(0)\neq f^{\ell-1}(\infty)$,
we must have $f^{\ell-1}(\infty)= - f^{\ell-1}(0)$.
Therefore,
\begin{align*}
\bigg( \prod_{i=1}^m \alpha_i &+ \prod_{i=1}^m \alpha'_i \bigg)^2
 = \prod_{i=1}^m \alpha_i^2 + \prod_{i=1}^m \alpha'_i{}^2 
+ 2 \prod_{i=1}^m \alpha_i \cdot \prod_{i=1}^m \alpha'_i \\
& = q_{\ell - 1} \cdot \bigg[ \bigg( \frac{ y - f^{\ell-1}(0) }{ y - f^{\ell-1}(\infty) } \bigg)
+ \bigg(\frac{ -y - f^{\ell-1}(0) }{ -y - f^{\ell-1}(\infty) }\bigg) + 2 \bigg] \\
& = q_{\ell - 1} \cdot \bigg[
\frac{ (y- f^{\ell -1}(0))^2 + (y+ f^{\ell -1}(0))^2 + 2(y- f^{\ell -1}(0))(y+ f^{\ell -1}(0))}
{y^2 - (f^{\ell -1}(\infty))^2} \bigg] \\
& = q_{\ell-1} \cdot \bigg( \frac{4 y^2 }{y^2 - (f^{\ell-1}(\infty))^2} \bigg)
= 4q_{\ell-1} \cdot \bigg( \frac{y^2 - 0^2}{y^2 - (f^{\ell-1}(\infty))^2} \bigg)
\\
%& = -4Cq_{\ell-1} \bigg( \frac{x-f(0)}{x-f(\infty)} \bigg)
%\bigg[ -C \bigg( \frac{x-f(0)}{x-f(\infty)} \bigg) + C \frac{f^{\ell}(0)-f(0)}{f^{\ell}(0)-f(\infty)}
%\bigg) \bigg]^{-1} \\
%&= 4q_{\ell-1} \bigg( \frac{ (x-f(0)) (f^{\ell}(0) - f(\infty))}
%{ (x-f(0))(f^{\ell}(0) - f(\infty)) - (x-f(\infty))(f^{\ell}(0) - f(0)) } \bigg) \\
%&= 4q_{\ell-1} \cdot \frac{ (x-f(0)) (f^{\ell}(0) - f(\infty))} { (f^{\ell}(0) - x) (f(\infty) - f(0))}
%= 4q_{\ell-1} \cdot \CR\big( f^{\ell}(0), f(\infty), x, f(0) \big).
& = 4q_{\ell-1} \cdot \CR\big( x, f(0), f^{\ell}(\infty), f(\infty) \big),
\end{align*}
where the third equality is because $f^{\ell-1}(\infty)= - f^{\ell-1}(0)$,
and the final equality is by equation~\eqref{eq:CRrewrite}
and the fact that $f(y)=x$.
\end{proof}

The expression $\prod \alpha_i + \prod\alpha'_i$ in Lemma~\ref{lem:Qn}
involves each of the $2m=2^{\ell-1}$ pairs of nodes at level $\ell$ of the tree.
The next lemma involves all the nodes at level $n\geq\ell$,
by partitioning them into $2^{n-\ell}$ such sets of $2^{\ell-1}$ pairs, and
taking the product of all the resulting expressions $\prod \alpha_i + \prod\alpha'_i$.
This product is important because its square is invariant under the
action of $M_{\ell}$ on the tree. More precisely,
even without squaring, it is invariant under
$(\ell,n)$-even cousins maps in $M_{\ell}$,
and it is sent to its negative by $(\ell,n)$-odd cousins maps in $M_{\ell}$.

\begin{lemma}
\label{lem:Rn}
Let $f(z)=(Az^2+B)/(z^2+C)\in K(z)$ and $\ell\geq 2$ as in Lemma~\ref{lem:Qn},
%with the critical points $0,\infty$ non-preperiodic but colliding at the $\ell$-th iterate.
with the critical points $0,\infty$ colliding at the $\ell$-th iterate,
and with $\infty$ not in the forward orbit of $0$.
Let $x\in\PKbar$ not be in the forward orbit of $f(\infty)$,
%Let $x\in\PKbar$ not in the forward orbit of $f(0)$ or $f(\infty)$,
%, write $f^{-1}(w)=\{\pm u\}$ for some $u\in \Kbar$,
%and assume that the backward orbits of $u$ and $-u$ contain no critical points.
let $n\geq\ell+1$, and define $m:= 2^{\ell -2}$ and $R:=2^{n-\ell}$.
Write the $2R$ points of $f^{-(n-\ell+1)}(x)$ as $\pm\beta_1,\ldots,\pm\beta_R$.
For each $j=1,\ldots, R$, write
\[ f^{-(\ell-1)}(\beta_j) = \{\pm \alpha_{j,1} ,\ldots, \pm \alpha_{j,m}\}
\quad\text{and}\quad
f^{-(\ell-1)}(-\beta_j) = \{\pm \alpha'_{j,1} ,\ldots, \pm \alpha'_{j,m}\} .\]
If necessary, reverse the names of $\alpha_{j,m}$ and $-\alpha_{j,m}$
as dictated by Lemma~\ref{lem:Qn} for $\beta_j$ in the role of $y$.
Then
\[ \prod_{j=1}^R \bigg( \prod_{i=1}^m \alpha_{j,i} + \prod_{i=1}^m \alpha'_{j,i} \bigg)^2
= r_n^2 \cdot \CR\big( x, f^{n-\ell+1}(0), f^n(\infty), f(\infty) \big) , \]
%= r_n \cdot \bigg( \frac{w- f^{n-\ell + 1}(0)}{w-f^n(0)} \bigg), \]
where
\[ r_n := (4q_{\ell-1})^{2^{n-\ell-1}}
\prod_{i=1}^{n-\ell} \bigg( \frac{f^{\ell+i-1}(\infty) - f(\infty)}{f^{i}(0) - f(\infty)}
\bigg)^{2^{n-\ell-i}} \in K^{\times} , \]
and where $q_{\ell-1}\in K^{\times}$ is as in Lemma~\ref{lem:Qn}.
%\[ r_n := (4q_{\ell-1})^{2^{n-\ell}}
%\prod_{i=0}^{n-\ell} \bigg( \frac{f^{\ell+i}(0) - f(\infty)}{f^{i+1}(0) - f(\infty)}
%\bigg)^{2^{n-\ell-i}} \in K , \]
%\[ r_{\ell} := 4q_{\ell-1} \cdot \bigg( \frac{f^{\ell}(0) - f(\infty)}{f(0) - f(\infty)} \bigg) \in K
%\quad \text{and} \quad
%r_k := r_{k-1}^2 \cdot \bigg(\frac{f^{n}(0) - f(\infty)}{f^{n-\ell+1}(0) - f(\infty)} \bigg) \in K
%\quad \text{for } k\geq \ell+1. \]
\end{lemma}

\begin{proof}
We proceed by induction on $n\geq \ell$.
Note that the lemma is only stated for $n\geq\ell+1$, but the formula for $r_n$
makes sense for $n=\ell$ and yields $r_{\ell}:=\sqrt{4 q_{\ell-1}}$, even though that
value may not lie in $K$.
With this definition of $r_{\ell}$,
the case $n=\ell$ is exactly the content of Lemma~\ref{lem:Qn}.
For the rest of the proof, then, consider $n\geq \ell+1$,
and suppose the statement holds for $n-1$.

\begin{figure}
\begin{tikzpicture}
\path[fill] (1.1,1.5) circle (0.1);
\path[fill] (5.3,1.5) circle (0.1);
\path[draw, dashed] (1.1,1.5) -- (3.2,0.5) -- (5.3,1.5);
\path[fill] (9.1,1.5) circle (0.1);
\path[fill] (13.3,1.5) circle (0.1);
\path[draw, dashed] (9.1,1.5) -- (11.2,0.5) -- (13.3,1.5);
\path[draw] (3.2,0.5) -- (7.2,-0.1) -- (11.2,0.5);
\path[fill] (3.2,0.5) circle (0.1);
\path[fill] (11.2,0.5) circle (0.1);
\path[fill] (7.2,-0.1) circle (0.1);
\node (x0) at (6.9,-0.3) {\small $x$};
\node (a0) at (2.9,0.3) {$y$};
\node (b0) at (11.6,0.3) {$-y$};
\path[fill] (0.6,2.0) circle (0.1);
\path[fill] (1.6,2.0) circle (0.1);
\path[fill] (4.8,2.0) circle (0.1);
\path[fill] (5.8,2.0) circle (0.1);
\path[draw] (0.6,2.0) -- (1.1,1.5) -- (1.6,2.0);
\path[draw] (4.8,2.0) -- (5.3,1.5) -- (5.8,2.0);
\node (be1) at (0.3,1.9) {\footnotesize $\beta_1$};
\node (mbe1) at (2.1,1.9) {\footnotesize $-\beta_1$};
\node (fbe1) at (0.6,1.3) {\footnotesize $f(\beta_1)$};
\node (dots1) at (3.3,2.0) {\footnotesize $\cdots$};
\node (dots1l) at (3.3,1.5) {\footnotesize $\cdots$};
\node (dots1h) at (3.3,3.4) {\footnotesize $\cdots$};
\node (be2) at (4.45,1.82) {\footnotesize $\beta_{R/2}$};
\node (mbe2) at (6.35,1.82) {\footnotesize $-\beta_{R/2}$};
\node (fbe2) at (5.95,1.3) {\footnotesize $f(\beta_{R/2})$};
\path[draw, dashed] (0.6,2.0) -- (0,3.4) -- (1.0,3.4) -- (0.6,2.0);
\path[draw, dashed] (1.6,2.0) -- (1.2,3.4) -- (2.2,3.4) -- (1.6,2.0);
\path[draw, dashed] (4.8,2.0) -- (4.2,3.4) -- (5.2,3.4) -- (4.8,2.0);
\path[draw, dashed] (5.8,2.0) -- (5.4,3.4) -- (6.4,3.4) -- (5.8,2.0);
\draw[<->] (0,3.6) -- (1.0,3.6);
\draw[<->] (1.2,3.6) -- (2.2,3.6);
\draw[<->] (4.2,3.6) -- (5.2,3.6);
\draw[<->] (5.4,3.6) -- (6.4,3.6);
\node (al1) at (0.5, 3.9) {\footnotesize $\alpha_{1,i}$};
\node (pal1) at (1.7,3.9) {\footnotesize $\alpha'_{1,i}$};
\node (al2) at (4.7, 3.9) {\footnotesize $\alpha_{R/2,i}$};
\node (pal2) at (5.9,3.9) {\footnotesize $\alpha'_{R/2,i}$};
\path[fill] (8.6,2.0) circle (0.1);
\path[fill] (9.6,2.0) circle (0.1);
\path[fill] (12.8,2.0) circle (0.1);
\path[fill] (13.8,2.0) circle (0.1);
\path[draw] (8.6,2.0) -- (9.1,1.5) -- (9.6,2.0);
\path[draw] (12.8,2.0) -- (13.3,1.5) -- (13.8,2.0);
\node (be3) at (8.2,1.82) {\footnotesize $\beta_{R/2+1}$};
\node (mbe3) at (10.25,1.82) {\footnotesize $-\beta_{R/2+1}$};
\node (fbe3) at (8.55,1.2) {\footnotesize $f(\beta_{R/2+1})$};
\node (dots2) at (11.4,2.0) {\footnotesize $\cdots$};
\node (dots2l) at (11.4,1.5) {\footnotesize $\cdots$};
\node (dots2h) at (11.4,3.4) {\footnotesize $\cdots$};
\node (be4) at (12.5,1.85) {\footnotesize $\beta_{R}$};
\node (mbe4) at (14.3,1.85) {\footnotesize $-\beta_{R}$};
\node (fbe4) at (13.9,1.3) {\footnotesize $f(\beta_{R})$};
\path[draw, dashed] (8.6,2.0) -- (8,3.4) -- (9.0,3.4) -- (8.6,2.0);
\path[draw, dashed] (9.6,2.0) -- (9.2,3.4) -- (10.2,3.4) -- (9.6,2.0);
\path[draw, dashed] (12.8,2.0) -- (12.2,3.4) -- (13.2,3.4) -- (12.8,2.0);
\path[draw, dashed] (13.8,2.0) -- (13.4,3.4) -- (14.4,3.4) -- (13.8,2.0);
\draw[<->] (8,3.6) -- (9.0,3.6);
\draw[<->] (9.2,3.6) -- (10.2,3.6);
\draw[<->] (12.2,3.6) -- (13.2,3.6);
\draw[<->] (13.4,3.6) -- (14.4,3.6);
\node (al3) at (8.3, 3.9) {\footnotesize $\alpha_{R/2+1,i}$};
\node (pal3) at (9.9,3.9) {\footnotesize $\alpha'_{R/2+1,i}$};
\node (al4) at (12.7, 3.9) {\footnotesize $\alpha_{R,i}$};
\node (pal4) at (13.9,3.9) {\footnotesize $\alpha'_{R,i}$};
\end{tikzpicture}
\caption{Lemma~\ref{lem:Rn}}
\label{fig:RnLemma}
\end{figure}

Write $f^{-1}(x)=\{\pm y\}$.
After re-indexing the $\beta_i$'s if necessary, we have
%$f^{-(n-\ell)}(y)=\{\pm\beta_1,\ldots,\pm\beta_{R/2} \}$
%and $f^{-(n-\ell)}(-y)=\{\pm\beta_{1+R/2},\ldots,\pm\beta_{R} \}$;
\[ f^{-(n-\ell)}(y)=\{\pm\beta_1,\ldots,\pm\beta_{R/2} \}
\quad\text{and}
f^{-(n-\ell)}(-y)=\{\pm\beta_{1+R/2},\ldots,\pm\beta_{R} \} \]
as in Figure~\ref{fig:RnLemma}.
Thus,
\begin{align*}
\prod_{j=1}^R \bigg( \prod_{i=1}^m & \alpha_{j,i} + \prod_{i=1}^m \alpha'_{j,i} \bigg)^2
=
\prod_{j=1}^{R/2} \bigg( \prod_{i=1}^m \alpha_{j,i} + \prod_{i=1}^m \alpha'_{j,i} \bigg)^2
\cdot
\prod_{j=1+R/2}^R \bigg( \prod_{i=1}^m \alpha_{j,i} + \prod_{i=1}^m \alpha'_{j,i} \bigg)^2
\\
&= r_{n-1}^2 \CR\big(y, f^{n-\ell}(0), f^{n-1}(\infty), f(\infty) \big) \cdot
r_{n-1}^2 \CR\big(-y, f^{n-\ell}(0), f^{n-1}(\infty), f(\infty) \big)
\\
&= \bigg[ r_{n-1}^2 \cdot \bigg( \frac{f^{n-1}(\infty) - f(\infty)}{f^{n-\ell}(0)-f(\infty)} \bigg) \bigg]^2
\cdot \bigg(
\frac{ (y - f^{n-\ell}(0))(-y - f^{n-\ell}(0))}{ (y - f^{n-1}(\infty))(-y - f^{n-1}(\infty))} \bigg)
\\
&= r_n^2 \cdot \frac{y^2 - f^{n-\ell}(0)^2}{y^2 - f^{n-1}(\infty)^2}
= r_n^2 \cdot \CR\big(x, f^{n-\ell+1}(0), f^n(\infty), f(\infty) \big),
\end{align*}
where the second equality is by the inductive hypothesis, the fourth
is by the definition of $r_n$, and the fifth
is by equation~\eqref{eq:CRrewrite} and the fact that $f(y)=x$.
\end{proof}

We are now prepared to prove our second main result, 
Theorem~\ref{thm:main2}, which we restate below in expanded form
as Theorem~\ref{thm:main2big}.
As we have assumed throughout this section, recall
that $K$ is a field of characteristic different from $2$, that
$f\in K(z)$ is a rational function of degree~$2$ with critical points
$\xi_1,\xi_2\in\PKbar$, and that $\delta\in K^{\times}$
is the discriminant of the minimal polynomial of $\xi_1$ over $K$.
% which we understand to be $\delta=1$ if $\xi_1\in\PK$.
%We fix $x_0\in\PK$, with associated extension fields $K_n$, $K_{\infty}$
%and Galois groups $G_n=\Gal(K_n/K)$ and $G_\infty(K_\infty/K)$.
We also define the quantities $\kappa_n\in L:=K(\sqrt{\delta})$
as in Definition~\ref{def:kappa}.

For any integer $n\geq 0$, recall from Definition~\ref{def:Mfinite}
that $M_{\ell,n}$ and $\widetilde{M}_{\ell,n}$ are
the quotients of the groups $M_{\ell}$ and $\widetilde{M}_{\ell}$
formed by restricting to the subtree $T_n$.
For ease of notation in stating Theorem~\ref{thm:main2big} below,
we will also sometimes denote as $M_{\ell}$ itself as $M_{\ell,\infty}$,
and $\widetilde{M}_{\ell}$ as $\widetilde{M}_{\ell,\infty}$.
As in Section~\ref{ssec:PinkDes}, we have
\[ M_{\ell,n}=\widetilde{M}_{\ell,n}=\Aut(T_n) \;\; \text{for } n\leq \ell-1,
\quad\text{and}\quad [\widetilde{M}_{\ell,n}:M_{\ell,n}]=2 \;\; \text{for } n\geq \ell. \]
As always, recall that when we say that two groups acting on a tree are isomorphic,
we mean that the isomorphism is equivariant with respect to the action.

\begin{thm}
\label{thm:main2big}
With notation and assumptions as stated just above,
let $N$ be either $\infty$ or a positive integer.
\begin{enumerate}
\item If $\delta$ is a square in $K$, then the following are equivalent:
\begin{enumerate}
\item No finite product $\dsps \kappa_{i_1} \cdots \kappa_{i_m}$
(for integers $1\leq i_1 < \cdots < i_m\leq N$ and $m\geq 1$) is a square in $K$.
\item $G_N\cong M_{\ell,N}$.
\end{enumerate}
\item If $\delta$ is not a square in $K$, then $\kappa_\ell \delta$ is
a square in $K$. In addition, if $N\geq \ell$, then the following are equivalent:
\begin{enumerate}
\item The only finite product $\dsps \kappa_{i_1} \cdots \kappa_{i_m}$
(for integers $1\leq i_1 < \cdots < i_m\leq N$ and $m\geq 1$) that is a square in $L$
is the single element $\kappa_{\ell}$.
\item $G_N\cong \widetilde{M}_{\ell,N}$.
\end{enumerate}
\smallskip
On the other hand, still assuming $\delta$ is not a square in $K$,
if instead $1\leq N\leq \ell-1$, then the following are equivalent:
\begin{enumerate}
\item No finite product $\dsps \kappa_{i_1} \cdots \kappa_{i_m}$
(for integers $1\leq i_1 < \cdots < i_m\leq N$ and $m\geq 1$) is a square in $K$.
\item $G_N\cong \widetilde{M}_{\ell,N}$.
\end{enumerate}
\end{enumerate}
\end{thm}

\begin{proof}
%[Proof of Theorem~\ref{thm:main2}]
%\label{prf: main2}
If either of the critical points $\xi_i$ lies in $f^{-n}(x_0)$ for some $n\geq 1$,
then Definition~\ref{def:kappa} yields $\kappa_n=0$, forcing each of conditions (1a) and (2a) to be false.
At the same time, two nodes at level $n$ of the tree must both correspond to $\xi_i$,
and hence any $\sigma\in G_{\infty}$ must act in exactly the same way on the
subtrees rooted at those two nodes.
Since there are elements of $M_{\ell}$ that act differently on any two such subtrees,
it follows that both versions of statement~(b) are also false.

Thus, we may assume for the remainder of the proof that
$x_0$ does not lie in the forward orbit of either $f(\xi_1)$ or $f(\xi_2)$.
In addition, as noted just before Definition~\ref{def:kappa},
by switching their roles if necessary,
we may further assume that $\xi_2$ is not in the forward orbit of $\xi_1$.

\medskip

\textbf{Case~1}. If $\delta$ is a square in $K$, then $\xi_1,\xi_2\in\PK$ are $K$-rational.
Hence, there is a linear fractional transformation $\nu\in\PGL(2,K)$
such that $\nu(0)=\xi_1$ and $\nu(\infty)=\xi_2$.
By Proposition~\ref{prop:coordchange}, changing coordinates by $\nu$
does not affect condition (a) of statement (1) of the theorem.
In addition, as noted at the start of the proof of that proposition,
the discriminant of the conjugate $\nu\circ f \circ \nu^{-1}$ is $\delta$ times a square in $K$,
and hence is itself a square in $K$.
Furthermore, this $K$-rational coordinate change does not change the field extensions
$K_n$ or $K_{\infty}$, and hence it also does not change the Galois groups $G_n$ or $G_{\infty}$
or their action on the tree of preimages.
Thus, condition (b) is also unaffected by this coordinate change.

Therefore, we may assume without loss of generality that
$\xi_1=0$ and $\xi_2=\infty$. Then $f(z)$ is a function of $z^2$.
Moreover, the fact that the critical points collide implies that $f$ cannot be a polynomial.
Thus, we have $f(z) = (Az^2 + B)/(z^2 +C)$ as in equation~\eqref{eq:ABCform},
with $AC-B\neq 0$.
(This last condition is because $\deg(f)=2$, and hence there is no cancellation.)

According to Theorem~\ref{thm:main1}, we may
label the tree so that $G_{\infty}$ is a subgroup of $M_{\ell}$,
since we have assumed that $\delta$ is a square in $K$.
%As in Section~\ref{sec:collideM},
%combining Corollary~\ref{cor:iterdisc}, Theorem~\ref{thm:discsquare},
%and Proposition~\ref{prop:DPQformula}
%shows that $\Delta(H_{\ell})\delta$ is a square in $K$.

\medskip
\noindent
\textbf{Claim}: 
For each integer $n\geq 1$, the quantity $\kappa_n\in K$ is a square in $K_n$.

\medskip

To prove the claim, first consider $1\leq n\leq \ell-1$.
Then the discriminant $\Delta(H_n)=\kappa_n$ is a square in $K_n$,
because all of the roots of $H_n$ are defined over $K_n$.
%By the same argument, the claim also holds for $n=\ell$ in the case
%that $\delta$ is not a square in $K$.

Next, consider $n\geq \ell+1$. Then
\[ \kappa_n = \CR\big(x_0,f^{n-\ell+1}(\xi_1), f^n(\xi_2), f(\xi_2)\big)
= \CR\big(x_0,f^{n-\ell+1}(0), f^n(\infty), f(\infty)\big) \]
is a square in $K_n$ by Lemma~\ref{lem:Rn}.
Indeed, with $x:=x_0$ in that lemma, the points $\alpha_{j,i}$ and $\alpha'_{j,i}$
lie in $f^{-n}(x_0)$, and $r_n\in K^{\times}\subseteq K_n^{\times}$.
Thus, the conclusion of the lemma yields that $\kappa_n$ is the square
of an explicit element of $K_n$.

The remaining possibility is that $n=\ell$.
% and $\delta$ is not a square in $K$.
Then the quantity $q_{\ell-1}\in K$ of Lemma~\ref{lem:Qn}
is a square in $K$ times
\[ \frac{f(\infty) - f^{\ell-1}(\infty)}{f(\infty) - f^{\ell-1}(0)}
= \frac{f(\xi_2) - f^{\ell-1}(\xi_2)}{f(\xi_2) - f^{\ell-1}(\xi_1)} \]
if $\ell\geq 3$, or times
\[ -4C = \Delta(-X^2-CY^2)) = \Delta(\theta_2 P - \eta_2 Q) \]
if $\ell=2$, since $(\eta_2,\theta_2)=(1,0)$ and $Q(X,Y)=X^2+CY^2$.
Thus, according to Lemma~\ref{lem:Qn} and Definition~\ref{def:kappa},
it follows that $\kappa_n$ is a square in $K_n$, proving our claim.

\medskip
\noindent
\textbf{1(a)$\boldsymbol{\Rightarrow}$1(b)}:
For this implication, since $G_{\infty}\cong\varprojlim G_n$,
it suffices to show that $G_n\cong M_{\ell,n}$ for each integer $0 \leq n \leq N$.
We proceed by induction on $n$.
The desired isomorphism holds trivially for $n=0$.

For arbitrary $1\leq n\leq N$, assume $G_{n-1}\cong M_{\ell,n-1}$.
By Theorem~\ref{thm:abelian}, the abelianization $G_{n-1}^{\textup{ab}}$ of $G_{n-1}$
is isomorphic to $\{\pm 1\}^{n-1}$.
By our claim, for each $i=1,\ldots,n-1$, the quantity $\kappa_i$
is a square in $K_i\subseteq K_{n-1}$.
If $\kappa_n$ were also a square in $K_{n-1}$, then we would have
\[ L'_n := K(\sqrt{\kappa_1},\ldots,\sqrt{\kappa_n}) \subseteq K_{n-1}.\]
However, it follows from condition 1(a) that $\Gal(L'_n/K)\cong \{\pm1\}^n$,
which is a strictly larger abelian group than the abelianization 
$G_{n-1}^{\textup{ab}}\cong\{\pm 1\}^{n-1}$. This is a contradiction,
since $K\subseteq L'_n\subseteq K_{n-1}$
and $L'_n/K$ is an abelian extension.

Thus, it must be that $\kappa_n$ is \emph{not} a square in $K_{n-1}$.
Again by our claim, we have $\sqrt{\kappa_n}\in K_n$, so there exists
\[ \sigma_n\in\Gal(K_n/K_{n-1}) \quad\text{with}\quad
\sigma_n(\sqrt{\kappa_n}) = -\sqrt{\kappa_n} .\]
That is, $\sigma_n$ acts trivially on the subtree $T_{n-1}$,
but it sends $\sqrt{\kappa_n}$ to its negative.

If $n\leq \ell-1$, then since $\kappa_n=\Delta(H_n)$,
we must have $\sgn_n(\sigma_n,x_0)=-1$.
By Proposition~\ref{prop:generate}, we have $G_n\cong \Aut(T_n)=M_{\ell,n}$, as desired.

If $n= \ell$, then writing
\[ f^{-n}(x_0)=\{\pm\alpha_1,\ldots,\pm\alpha_m\} \cup
\{\pm\alpha'_1,\ldots,\pm\alpha'_m\} \]
as in Lemma~\ref{lem:Qn}, we have that
\[ \pi_n:=\prod_{i=1}^m \alpha_i + \prod_{i=1}^m \alpha'_i \in K_n \]
is $\sqrt{\kappa_n}$ times an element of $K^{\times}\subseteq K_{n-1}^{\times}$.
Thus, we must have $\sigma_n(\pi_n)=-\pi_n$.
Since $\sigma_n$ fixes each node $f(\alpha_i)$ and $f(\alpha'_i)$ at level $n-1$,
it follows that $\sigma_n$ acts as an odd permutation of both 
$\{\pm\alpha_1,\ldots,\pm\alpha_m\}$
and $\{\pm\alpha'_1,\ldots,\pm\alpha'_m\}$.
That is, $\sigma_n$ is an $(\ell,\ell)$-odd cousins map.
By Theorem~\ref{thm:generate}, we again have $G_n\cong M_{\ell,n}$.

Finally, if $n\geq \ell+1$, then writing
\[ f^{-n}(x_0)=\bigcup_{j=1}^R \Big( \{\pm\alpha_{j,1},\ldots,\pm\alpha_{j,m}\} \cup
\{\pm\alpha'_{j,1},\ldots,\pm\alpha'_{j,m}\} \Big )\]
as in Lemma~\ref{lem:Rn}, we have that
\[ \pi_n:=\prod_{j=1}^R\big( \prod_{i=1}^m \alpha_{j,i} + \prod_{i=1}^m \alpha'_{j,i} \big) \in K_n \]
is $\sqrt{\kappa_n}$ times an element of $K^{\times}\subseteq K_{n-1}^{\times}$.
Thus, we again have $\sigma_n(\pi_n)=-\pi_n$,
while $\sigma_n$ also fixes each node $f(\alpha_{j,i})$ and $f(\alpha'_{j,i})$ at level $n-1$.
Hence, for an odd number of indices $j$,
$\sigma_n$ acts as an odd permutation of both 
$\{\pm\alpha_{j,1},\ldots,\pm\alpha_{j,m}\}$
and $\{\pm\alpha'_{j,1},\ldots,\pm\alpha'_{j,m}\}$.
That is, $\sigma_n$ is an $(\ell,n)$-odd cousins map,
and therefore by Theorem~\ref{thm:generate}, we conclude
that $G_n\cong M_{\ell,n}$, completing our induction.

\medskip
\noindent
\textbf{1(b)$\boldsymbol{\Rightarrow}$1(a)}:
We prove the contrapositive of this implication.
Consider a product $\kappa_{i_1} \cdots \kappa_{i_m}$ that is a square in $K$,
with $1\leq i_1 < \cdots < i_m \leq N$ and $m\geq 1$.
Without loss, assume $n:=i_m$ is the smallest index for which such a product exists.

Because no nonempty product of $\kappa_i$'s with $i<n$ is a square in $K$,
the argument from the previous implication shows that $G_{n-1}\cong M_{\ell,n-1}$.
By the claim within that implication, each $\kappa_i$ is
a square in $K_i^{\times}\subseteq K_{n-1}^{\times}$. That is, we have
\[ L'_{n-1} := K(\sqrt{\kappa_1},\ldots,\sqrt{\kappa_{n-1}}) \subseteq K_{n-1}.\]
Since the product $\kappa_{i_1} \cdots \kappa_{i_m}$ is a square
in $K^{\times}\subseteq K_{n-1}^{\times}$, it follows that
\begin{equation}
\label{eq:kappasquare}
\sqrt{\kappa_n}=\sqrt{\kappa_{i_m}}\in L'_{n-1}\subseteq K_{n-1}.
\end{equation}

If $n\leq \ell-1$, then $\kappa_n=\Delta(H_n)$.
Therefore, equation~\eqref{eq:kappasquare} implies that every
$\sigma\in\Gal(K_n/K_{n-1})$ acts as an even permutation of the nodes at level $n$
of the tree. But $M_{\ell,n}=\Aut(T_n)$ includes automorphisms
that fix all the nodes below level $n$ and yet are odd at level $n$.
Thus, $G_n \subsetneq M_{\ell,n}$.
In fact, because $M_{\ell,n}$ is a finite group, it follows
that $G_n$ is not isomorphic to $M_{\ell,n}$,
even if we were to label the tree differently.
%(even if we were to label the tree differently).
%and hence $G_{\infty}\not\cong M_\ell$.

If $n\geq \ell$, then with notation as in the previous implication,
we also have that $\pi_n$ is a square in $K_{n-1}$,
and hence every $\sigma\in\Gal(K_n/K_{n-1})$ is an $(\ell,n)$-even cousins map.
However, $M_{\ell,n}$ includes odd cousins maps that fix the nodes below level $n$.
Thus, as in the $n\leq \ell-1$ case above, it follows that
$G_n \subsetneq M_{\ell,n}$, and hence
that $G_n$ is not isomorphic to $M_{\ell,n}$.
%and hence that $G_{\infty}\not\cong M_\ell$.

\medskip
\noindent
\textbf{Case 2.}
As in Section~\ref{sec:collideM}, we have that
$\kappa_{\ell}\delta=\Delta(H_{\ell})\delta$ is a square in $K$, by
Corollary~\ref{cor:iterdisc}, Theorem~\ref{thm:discsquare},
and Proposition~\ref{prop:DPQformula}.
Thus, we have proven the first desired statement.

Moreover, if $N\leq \ell-1$, then $\widetilde{M}_{\ell,N}=M_{\ell,N}$,
and so we are already done by Case~1.
Therefore, we may assume for the remainder of the proof that $N\geq\ell$.

As in Case~1, we have $\sqrt{\Delta(H_{\ell})}\in K_\ell$,
and hence $L = K(\sqrt{\delta})\subseteq K_n$ for every $n\geq \ell$.
Thus, we may define $G'_n:=\Gal(K_n/L)$ for each integer $n\geq \ell$, and
$G'_{\infty}:=\Gal(K_{\infty}/L)$.
By Theorem~\ref{thm:main1} and Corollary~\ref{cor:main1},
the Galois groups $G_{n}$ and $G'_{n}$
are equivariantly isomorphic to subgroups of
$\widetilde{M}_{\ell,n}$ and  $M_{\ell,n}$, respectively, for every $1\leq n\leq \infty$.

\medskip
\noindent
\textbf{2(a)$\boldsymbol{\Rightarrow}$2(b)}:
Since $\kappa_{\ell}\delta$ is a square in $K$ but $\delta$ is not,
condition~2(a) implies that no nontrivial product $\kappa_{i_1}\cdots\kappa_{i_m}$
with $1\leq i_1 < \cdots < i_m\leq \ell$ is a square in $K$.
By the same argument as in the 1(a)$\Rightarrow$1(b) proof above,
it follows that
% $G_{n}\cong\Aut(T_n) = \widetilde{M}_{\ell,n}$ if $n< \ell$, and that
%$G_{\ell}\cong\Aut(T_\ell) = \widetilde{M}_{\ell,\ell}$ if $n\geq \ell$.
$G_{\ell}\cong\Aut(T_\ell) = \widetilde{M}_{\ell,\ell}$.

As noted above, we also have $L\subseteq K_{\ell}$.
Since $[L:K]=2$, it follows that $G'_{\ell}$ is a subgroup of $\Aut(T_{\ell})$
of index~2. In addition, since $\kappa_{\ell}=\Delta(H_{\ell})$ is a square in $L$,
every $\sigma\in G'_{\ell}$ acts as an even permutation of the nodes at the $\ell$-th
level of the tree. Thus, $G'_{\ell}$ must be isomorphic to the set of elements of $\Aut(T_{\ell})$
that are even at the $\ell$-th level; that is, $G'_{\ell}\cong M_{\ell,\ell}$.

Applying the same inductive argument as in the $n\geq\ell+1$ portion of the
1(a)$\Rightarrow$1(b) proof, it follows that $G'_n\cong M_{\ell,n}$
for all $1\leq n \leq N$.
%Thus, $G'_{\infty}$ is isomorphic to $M_{\ell}$.
%By Theorem~\ref{thm:main1}
Thus, $G_{n}$ is isomorphic to a subgroup of $\widetilde{M}_{\ell,n}$
that contains $G'_{n}\cong M_{\ell,n}$ as a subgroup of index $[L:K]=2$.
Therefore, $G_{n}$ is isomorphic to the whole group $\widetilde{M}_{\ell,n}$
for all such $n$, including $n=N$.
%By Theorem~\ref{thm:main1}, then, $G_{\infty}$ is isomorphic to a subgroup of $\widetilde{M}_{\ell}$
%that contains $G'_{\infty}\cong M_{\ell}$ as a subgroup of index $[L:K]=2$.
%Hence, $G_{\infty}$ is isomorphic to the whole group $\widetilde{M}_{\ell}$.

\medskip

\noindent
\textbf{2(b)$\boldsymbol{\Rightarrow}$2(a)}:
Suppose there is a product
$\kappa_{i_1} \cdots \kappa_{i_m}$ that is a square in $K$,
with
\[ 1\leq i_1 < \cdots < i_m\leq N \quad\text{ and } m\geq 1, \]
and such that some $i_j$ is not $\ell$.
Thus, even if $\kappa_{\ell}$ appears in the product, we may remove it,
leaving a nontrivial product $\kappa_{i_1} \cdots \kappa_{i_m}$ that does not include $\kappa_{\ell}$,
and which is not a square in $L$.

By Case~1, it must be that $G'_{N}$ is \emph{not} isomorphic to $M_{\ell,N}$.
That is, $G'_{N}$ is isomorphic to a proper subgroup of $M_{\ell,N}$, so that
$[M_{\ell,N}:G'_{N}]\geq 2$.
Since $[G_{N}:G'_{N}] = [L:K]=2$, and $[\widetilde{M}_{\ell,N}:M_{\ell,N}]=2$,
it follows that $[\widetilde{M}_{\ell,N}:G_{N}]\geq 2$.
In particular, $G_{N} \not\cong \widetilde{M}_{\ell}$.
%That is, $G'_{n}$ is isomorphic to a proper subgroup of $M_{\ell,n}$, so that
%$[M_{\ell}:G'_{\infty}]\geq 2$.
%Since $[G_{n}:G'_{n}] = [L:K]=2$, and $[\widetilde{M}_{\ell}:M_{\ell}]=2$,
%it follows that $[\widetilde{M}_{\ell}:G_{\infty}]\geq 2$.
%In particular, $G_{\infty} \not\cong \widetilde{M}_{\ell}$.
\end{proof}

\textbf{Acknowledgments}.
The first author gratefully acknowledges the support of NSF grant DMS-2101925.
The authors thank the anonymous referee for their careful reading of the original
manuscript of the paper, and for their suggestions, which greatly improved the
exposition.

\bibliographystyle{amsalpha}

\end{document}